\documentclass[11pt,reqno]{amsart}
\numberwithin{equation}{section}
\usepackage[dvipsnames]{xcolor}
\usepackage{tikz}
\usetikzlibrary{snakes}
\usetikzlibrary{patterns}
\usepackage{amsmath,amsfonts,amssymb,amsbsy,amsopn,amsthm,eucal, mathrsfs}
\usepackage{enumerate}
\usepackage{dsfont}
\usepackage{graphicx}   % for figures
\usepackage[bookmarks=true, pdfstartview=FitH, colorlinks=true,citecolor=blue, linkcolor=blue]{hyperref}
\usepackage{accents}
\usepackage{enumerate}
\usepackage{enumitem}
\usepackage{color}

\usepackage{cite}
\newcommand{\Ric}{{\rm Ric}}

\newcommand{\diam}{{\rm diam}}

\newcommand{\Alex}{\text{Alex\,}}

\newcommand{\Alexnk}{\text{Alex}^n(\kappa)}

\newcommand{\dN}{\mathds{N}}

\newcommand{\dR}{\mathds{R}}
\newcommand{\dS}{\mathds{S}}

\newcommand{\cH}{\mathcal{H}}

\newcommand{\cK}{\mathcal{K}}
\newcommand{\cKS}{\mathcal{KS}}
\newcommand{\cL}{\mathcal{L}}

\newcommand{\cS}{\mathcal{S}}

\newcommand{\cW}{\mathcal{W}}

\newcommand{\ang}[3]{\measuredangle\left({#1}\,_{#3}^{#2}\right)}
\newcommand{\tang}[3]{\tilde\measuredangle\left({#1}\,_{#3}^{#2}\right)}

\newcommand{\Vol}[1]{\text{Vol}\left(#1\right)}

\newcommand{\dsp}{\displaystyle}

\newtheorem*{Theorem A}{Theorem A}
\newtheorem*{Theorem B}{Theorem B}
\newtheorem*{Theorem C}{Theorem C}
\newtheorem*{Theorem D}{Theorem D}
\newtheorem*{Theorem E}{Theorem E}

\newtheorem{theorem}{Theorem}[section]

\newtheorem{lemma}[theorem]{Lemma}
\newtheorem{corollary}[theorem]{Corollary}
\newtheorem{conjecture}[theorem]{Conjecture}
\newtheorem{definition}[theorem]{Definition}
\newtheorem{example}[theorem]{Example}
\newtheorem{remark}[theorem]{Remark}

\makeatletter
\def\smalloverbrace#1{\mathop{\vbox{\m@th\ialign{##\crcr\noalign{\kern3\p@}%
  \tiny\downbracefill\crcr\noalign{\kern3\p@\nointerlineskip}%
  $\hfil\displaystyle{#1}\hfil$\crcr}}}\limits}
\makeatother

\begin{document}
%\title{sumable singularity}

\title{Quantitative estimates on the $C^2$-singular sets in Alexandrov spaces}

\author[N.~Li]{Nan Li}
\address{N. Li, Department of Mathematics, The City University of New York - NYC College of
Technology, 300 Jay St., Brooklyn, NY 11201}
\email{NLi@citytech.cuny.edu}
\thanks{Nan Li was partially supported by PSC-CUNY \#65397-00 53.}

%\date{\today}
\maketitle

%\begin{abstract}
%For singular spaces, a common way to measure singularity is to use tangent cones. However, examples (double disk) show that the metric may still not be smooth even if the tangent cone is isometric to $\dR^n$. In this paper we define a notion to capture and measure such singularities and prove that the integral of all such singularities in a unit ball, with respect to the co-dimension 1 Hausdorff measure, is bounded by the dimension and curvature lower bound on Alexandrov spaces. It's closely related to the scalar curvature measure on smooth manifolds. 
%\end{abstract}

\begin{abstract}
%	The total amount of uncontrollability or singularity could be controllable under certain conditions. 
%The total quality of disaster is kindly expected to be controllable if not preventable. 
 
%Yau's conjecture states that  $\int_{B_1}scal \,\operatorname d {vol}_{g}\le C(n,\kappa)$ in any $n$-dimensional Riemannian manifold $(M, g)$ with Ricci curvature $Ric_M\ge (n-1)\kappa$. This was proved by A. Petrunin if the sectional curvature $sec_M\ge\kappa$. Similar questions can also be asked in the context of length metric spaces. For instance, A. Naber and the author showed that the quantitative $(n-2)$-dimensional Hausdorff measure estimate $\mathcal H^{n-2}(\{\theta\ge\epsilon\}\cap B_1)<C(n,\kappa,\epsilon)$ holds for any $n$-dimensional Alexandrov space with curv $\ge \kappa$ and without boundary, where $\theta(x)$ measures the deviation of the tangent cone at $x$ from being isometric to $\mathds R^n$. A conjecture of the integral bound $\int_{B_1}\theta(x)\cdot \operatorname d\mathcal H^{n-2} < C(n, \kappa)$ can also be formulated. Based on A. Petrunin's work, this conjecture holds if the space is a non-collapsed Gromov-Hausdorff limit of manifolds with lower sectional curvature bounds.
The total disaster may be controllable if not preventable. We will explore this phenomenon for singularities in metric spaces. A point in an $n$-dimensional Alexandrov space is called regular if its tangent cone is isometric to $\mathds R^n$. Examples show that not every regular point is smooth, and the non-smooth points, away from the boundary, can have co-dimension 1. In this paper, we define a non-negative function $\mathcal K(x)$, which quantitatively measures the extent of the point $x$ from being $C^2$. The so-called $C^2$-singular points are identified as the set where $\mathcal K>0$. We show that $\int_{B_r(p)} \mathcal K(x)\, \operatorname d\mathcal H^{n-1}\le c(n,\kappa,\nu)r^{n-2}$ for any $n$-dimensional Alexandrov space $(X,p)$ with curv $\ge \kappa$ and $\operatorname{Vol}\left(B_1(p)\right)\ge\nu>0$. This leads to the Hausdorff dimension estimate $\dim_\mathcal H\{\mathcal K>0\}\le n-1$, and the quantitative Hausdorff measure estimate $\mathcal H^{n-1}\left(\{\mathcal K>\epsilon\}\cap B_r(p)\right)\le \epsilon^{-1}\cdot c(n,\nu)r^{n-2}$. These results also make progress on Naber's conjecture on the convergence of curvature measures. 

The measure $\mathcal K(x)\, \operatorname d\mathcal H^{n-1}$ on Alexandrov spaces can be viewed as the counterpart of the curvature measure $scal \,\operatorname d {vol}_{g}$ on smooth manifolds.  We also show that if $n$-dimensional Alexandrov spaces $X_i$ Gromov-Hausdorff converge to a smooth manifold with no boundary without collapsing, then $\mathcal K_i\, \operatorname d\mathcal H^{n-1}\to 0$ as a measure. 

\end{abstract}

\tableofcontents

%For Theorem B
%\blue{maybe convergence holds for different measure defined on $X\setminus\cKS(X)$. Maybe something using intrinsic volume? Petrunin's limit tensor is equivalent to the intrinsic volume thing?} 

%For (\ref{l:int.formula.K.e3}) 
%\blue{(\ref{l:int.formula.K.e3}) is not equality, because $\Theta(t)$ is not monotonic. Counterexample, flow of corner can become straight. May work with more carefully designed flow.}

%Discuss with Petrunin
%\begin{itemize}
%	\item Is K=tensor in Petrunin's paper? 
%	\item $H^{n-2}$-estimate for $\cS^{n-2}$ 
%	\item Definition of $R dH^n$ by intrinsic volume. Theorem B to (1.9).
%	\item Does k=0 imply smooth?
%	\item Convergence and finiteness of (co-dimension 2) intrinsic volume, using epsilon-frames and convex embedding
%	\item relation to Petrunin's question: how many points on 2-dim Alex space have two geodesics passing through?
%	\item Relation to metric measure boundary
%\end{itemize}
%1. Is K=tensor in Petrunin's paper? 
%2. H^{n-2}-estimate for \cS^{n-2} 
%3. Theorem B to (1.11).
%4 Does k=0 imply smooth?
%5 Definition of R dH^n by intrinsic volume
%6. Convergence and finiteness of intrinsic volume, using epsilon-frames and convex embedding
%7. relation to Petrunin's question: how many points on 2-dim Alex space have two geodesics passing through?
%8. Relation to metric measure boundary

\section{Introduction}\label{s:intro}

%Let $X$ be an Alexandrov space with curv $\ge 0$ and Gromov-Hausdorff close to unit ball in $\dR^n$. Let $\gamma\subset\partial X$ be a curve with finite length. Consider the gradient flow $\Phi$ of the distance function $d_{\partial X}$. Let $\gamma_s=\Phi_s(\gamma)$ be the flowed curve of $\gamma$ and $L(\gamma_s)$ be its length. Is there any lower bound of $\frac{d}{ds}L(\gamma_s)\mid_{s=0}$? In another word, does this lower bound depend only on $X$ or it also depends on $\gamma$?
%
%I guess if this is true for length of curves, the it should also be true for Hausdorff measure of subsets.

In 1982, S.T. Yau proposed the following conjecture on the concentration of curvatures.

\begin{conjecture}[S.T.~Yau \cite{Yau82}]\label{c:Yau}
	For any $n$-dimensional manifolds $(M, g)$ with Ricci curvature $Ric_M\ge (n-1)\kappa$, we have a prior $L^1$-bound for the scalar curvature
	$\dsp\int_{B_1}scal \,\operatorname d {vol}_{g}\le c(n,\kappa)$.
\end{conjecture}

It posits that if the Ricci curvature is bounded from below, then the regions where the scalar curvature is positively concentrated do not have a large volume. One can also compare it to the Gauss-Bonnet Theorem. Despite some progress (see the following theorem) in proving the conjecture under stronger conditions, it remains an open problem in mathematics.

\begin{theorem}[Petrunin \cite{Pet09}]\label{t:int.scal}
	Conjecture \ref{c:Yau} holds if the sectional curvature $\sec_M\ge \kappa$.
\end{theorem}

One can ask similar questions on length metric spaces with synthetic lower curvature bounds. Let $\Alex^n(\kappa)$ be the collection of $n$-dimensional Alexandrov spaces with curvature $\ge \kappa$ and $X\in\Alex^n(-1)$. Let $\cS^k(X)$ be the $k$-stratum singular set of points in $X$ whose tangent cones do not split off $\dR^{k+1}$ isometrically. Due to Burago-Gromov-Perelman, and Cheeger-Colding's work, we have that $\dim_{\cH}(\cS^k(X))\le k$ for both Alexandrov spaces \cite{BGP} and Ricci limit spaces \cite{CC97-I}. Though the dimension estimate gives a nice control for the singular set, $\cS^k(X)$ may still have infinite $k$-dimensional Hausdorff measure (see \cite{OtsShi94}). Since a quantitative estimate of $\cS^k(X)$ is needed in geometric analysis, we have the following result in the further study. Let $\cS^k_\epsilon(X)$ denote the set of points $x\in X$ whose tangent cone $T_x(X)$ is $\epsilon$-away from splitting off $\dR^{k+1}$. That is, for any $x\in\cS^k_\epsilon(X)$, either $T_x(X)$ doesn't split off $\dR^k$ or $T_x(X)=\dR^k\times C(\Sigma^{n-k-1}_x)$ and the  Gromov-Hausdorff distance $d_{GH}(\Sigma^{n-k-1}_x, \dS_1^{n-k-1})\ge\epsilon$. The following Hausdorff measure estimate was given in \cite{LiNab20} by A. Naber and the author: 
\begin{align}
	\mathcal H^k\big(\cS^k_{\epsilon}(X) \cap B_1\big)\le C(n,\epsilon).
\end{align}
%In particular, $\dim_\cH(\cS^k(X))\le k$. 
It was conjectured \cite{LiNab20} that the upper bound $C(n,\epsilon)$ grows in an order of $\epsilon^{1-(n-k)}$, as $\epsilon\to 0$. 
\begin{conjecture}[Li-Naber \cite{LiNab20}]\label{conj1}
	For any $X\in\Alex^n(-1)$, we have 
	\begin{align}
		\cH^{k}(\cS^{k}_\epsilon(X)\cap B_1)<C(n)\epsilon^{1-(n-k)}.
		\label{conj1.e1}
	\end{align}
\end{conjecture}
We can re-write inequality (\ref{conj1.e1}) as 
\begin{align}
	\epsilon^{n-k-1}\cdot\cH^{k}(\cS^{k}_\epsilon\cap B_1)<C(n)
\end{align}
and formulate a much stronger conjecture in a summable form \cite{LiNab20} or the following integral form. This is based on the observation of Gauss-Bonnet Theorem and some higher dimensional examples. 
%\begin{conjecture}[Li-Naber \cite{LiNab20}]\label{conj2}
%  For any $X\in\Alex^n(-1)$ and $\epsilon_i =2^{-i}$, we have \newline $$\sum_{i=0}^\infty \epsilon_{i+1}^{n-k-1}\cdot\cH^{k}\Big(\big(\cS^{k}_{\epsilon_{i+1}}\setminus\cS^{k}_{\epsilon_{i}} \big)\cap B_1\Big)<C(n).$$
%\end{conjecture}
%Let $\theta_k(x)=d_{GH}(\Sigma^{n-k-1}_x, \dS_1^{n-k-1})$ or simply 
Let $\theta_k(x)=2\pi$ for $x\in\cS^{k-1}(X)$, and $\theta_k(x)=2\pi-\diam(\Sigma^{n-k-1}_x)$ for $x\in\cS^k(X)\setminus\cS^{k-1}(X)$ whose tangent cone $T_x(X)=\dR^k\times C(\Sigma^{n-k-1}_x)$. Here  $\theta_k(x)$ measure the $k$-singularity of $x$. 
%then the above conjecture can be formulated to the following integral form. 

\begin{conjecture}\label{c:int-sing}
	For any $X\in\Alex^n(-1)$, we have $\displaystyle \int_{B_1}\theta_k^{n-k-1}(x)\cdot \operatorname d\cH^k < C(n)$.
%	\begin{align}
%		\int_{B_1}\theta_k^{n-k-1}(x)\cdot \operatorname d\cH^k < C(n).
%		\label{c:int-sing.e1}
%	\end{align}
In particular, for $k=n-2$, we have 
$\dsp\int_{B_1}\theta_{n-2}(x)\cdot \operatorname d\cH^{n-2} < C(n)$.
\end{conjecture}

%$$\mathcal{K}(x) = \sup_{\textbf{v},\textbf{w}\in T_xM, \textbf{v} \neq \textbf{w}}\frac{|\nabla_{\textbf{v}}\nabla_{\textbf{w}}\textbf{g}|}{|\textbf{g}|}$$

Note that $\cS(X)=\cup_{k=0}^{n-1}\cS^k(X)$ contains all geometric singular points, as all tangent cones at $x\notin \cS(X)$ are isometric to $\dR^n$ by definition. However, the metric away from $\cS(X)$ may still be non-smooth. For instance, if we double an $n$-dimensional smooth manifold $M$ with boundary, and glue them along their boundaries, then the glued metric may not be smooth at any of the glued point. The glued points have co-dimension 1 and their tangent cones are all isometric to $\dR^n$ if the boundary $\partial M$ is smooth. In particular, these points can't be differentiated from the smooth points by their tangent cones. 
In this paper, we define a singularity function $\cK(x)\ge 0$ which measures the extent to which the metric is from being locally $C^2$ at a point $x$. See Section \ref{s:defn} for details and examples. In particular, we have that every point in $\cKS(X)=\{x\in X\colon \cK(x)>0\}$ is not locally $C^2$. The larger $\cK(x)$ is, the more curving the space would look like near $x$. Let $\cKS_\epsilon(X)=\{x\in X\colon \cK(x)\ge\epsilon\}$. We have the following quantitative estimates. 

\begin{Theorem A}\label{t:main}
	Let $(X,p)\in\Alex^n(-1)$ with $vol(B_1(p))\ge \nu >0$. Then for any $B_r\subseteq B_1(p)$, we have 
	\begin{align}
		\int_{B_r} \cK(x)\, \operatorname d\cH^{n-1}\le c(n,\nu)r^{n-2}.
		\label{t:main.e1}
	\end{align}
	In particular, we have the Hausdorff measure estimate
	\begin{align}
		\cH^{n-1}\left(\cKS_\epsilon(X)\cap B_r\right)\le \epsilon^{-1}\cdot c(n,\nu)r^{n-2}
	\end{align}
	for $\epsilon>0$, and the Hausdorff dimension estimate
	\begin{align}
		\dim_\cH(\cKS(X))\le n-1. 
	\end{align}
\end{Theorem A}

It turns out that $\cK(x)\, \operatorname d\cH^{n-1}$ can also be viewed as a synthetic curvature measure on Alexandrov spaces, as a counterpart of $scal \,\operatorname d {vol}_{g}$ on smooth manifolds. Suppose $(M_i,g_i)$ Gromov-Hausdorff converges to a limit space $(X, d)$, and $scal_i$ satisfy the integral bounds in Conjecture \ref{c:Yau}. Passing to a sub-sequence, as a signed measure, $scal_i \,\operatorname d {vol}_{g_i}$ converges to a measure on $X$. A.~Naber proposed the following conjecture. 

\begin{conjecture}[Conjecture 2.20 in \cite{Nab20}]\label{c:measure conv}
	Let $(M_i, g_i)$ be a sequence of $n$-dimensional manifolds with lower Ricci or sectional curvature bounds. If $M_i$ Gromov-Hausdorff converge to $X$ without collapsing, then as measure, 
	$$ scal_i \,\operatorname d {vol}_{g_i} \to {\operatorname R}\,\operatorname d \cH^n + \Phi\,\operatorname d \cH^{n-1} +\theta \,\operatorname d \cH^{n-2},
	$$
	where ${\operatorname R}$, $\Phi$ and $\theta$ are locally $L^1$-functions with respect to $\cH^n$, $\cH^{n-1}$ and $\cH^{n-2}$ respectively. The function $\Phi(x)$ is supported on an $(n-1)$-rectifiable subset. The function $\theta(x)=\theta_{n-2}(x)$ is defined as in Conjecture \ref{c:int-sing}, which is supported on the top stratum of the $(n-2)$-rectifiable singular set $\cS(X)=\cS^{n-2}(X)$.
\end{conjecture}
%The limit measure vanishes on any $\cH^{n-2}$-zero measure subset, because $\int scal$ will be re-scaled by $r^{n-2}$ if the metric $d$ is re-scaled to $rd$. \red{Petrunin's continuity theorem} 

The above conjecture is true in the following special cases. 

\begin{theorem}[Petrunin \cite{Pet03}]\label{t:cont.sec}
	Let $(M_i,g_i)$ be a sequence of $n$-dimensional manifolds with $\sec_{M_i}\ge \kappa$ and  $(M_i,p_i, g_i)\overset{d_{GH}}\longrightarrow(M, p, g)$, as $i\to\infty$. If $(M,g)$ is an $n$-dimensional smooth manifold, then 
	\begin{align}
		\dsp\lim_{i\to\infty}\int_{B_1(p_i)}scal_i \,\operatorname d {vol}_{g_i}=\int_{B_1(p)}scal \,\operatorname d {vol}_{g}.
		\label{t:cont.sec.e1}
	\end{align}
\end{theorem}

\begin{theorem}[Lebedeva-Petrunin \cite{Pet22}]\label{t:Pet22}
  Let $\{(M_i,g_i)\}$ be a sequence of $n$-dimensional Riemannian manifolds with $\sec_{M_i}\ge \kappa$ and $M_i\overset{d_{GH}}\longrightarrow X$ without collapsing, then the curvature tensors of $M_i$ weakly converge to a measure-valued tensor ${\operatorname R}\, \operatorname d \cH^n + \Phi \,\operatorname d \cH^{n-1} +\theta \,\operatorname d \cH^{n-2}$ on $X$, where $\theta(x)=\theta_{n-2}(x)$ is defined as in Conjecture \ref{c:int-sing}. 
\end{theorem}
Theorem \ref{t:int.scal} and Theorem \ref{t:Pet22} imply that Conjecture \ref{c:int-sing} is true for $k=n-2$ and $X$ being a non-collapsed limit of manifolds with lower sectional curvature bound. When comparing Theorem \ref{t:Pet22} with Theorem \ref{t:cont.sec}, we see that as a limit, $\Phi$ should be supported on the set of non-smooth points. However, this is unknown based on \cite{Pet22}. It remains unclear the relation between the limit measure $\Phi \, \operatorname d \cH^{n-1}$ and the underlying geometry of the limit space $X$. For instance, one can ask whether $\Phi$ can be defined on the singular space without or independent on the smoothing procedure $M_i\to X$. In our case, the $C^2$-singularity function $\cK(x)$ solely depends on the metric $X$, which also measures the bending of the space near $x$. This indicates that $\cK \,\operatorname d \cH^{n-1}$ is equivalent to $\Phi \,\operatorname d \cH^{n-1}$, and this is true in the simple doubling examples. In light of this, $\cK \, \operatorname d \cH^{n-1}$, as well as $\theta_{n-2} \, \operatorname d \cH^{n-2}$, can be viewed as a synthetic curvature measure on the non-smooth sets in Alexandrov spaces. The definition of ${\operatorname R} \, \operatorname d \cH^n$ is not clear to the author at this point, although in the smooth case, ${\operatorname R}$ can be chosen as the scalar curvature.

By introducing the notion of $\operatorname R$, $\cK$ and $\theta$, we can define a synthetic curvature measure on Alexandrov space $X\in\Alex^n(-1)$ as follows:
\begin{align}
	\mu_X = {\operatorname R}\,\operatorname d \cH^n + \cK\,\operatorname d \cH^{n-1} +\theta\,\operatorname d \cH^{n-2}.
\end{align}

The first question one may ask is if $\mu_X$ is locally finite. 
\begin{conjecture}
	If $X\in\Alex^n(-1)$, then $\int_{B_1}\mu_X\le C(n)$. 
\end{conjecture}

On can also ask convergence questions similar to Conjecture \ref{c:measure conv}.

\begin{conjecture}
	If $X_i\in\Alex^n(-1)$ Gromov-Hausdorff converge to $X$ without collapsing, then $\mu_{X_i}\to \mu_X$. 
\end{conjecture}

Regarding this conjecture, we have the following partial result. 

%First, it's easy to find that $\int_{B_1}\cK_i\, \operatorname d\cH^{n-1}$ that such a continuity result is not true in general. In Example \ref{e:2disk}, $\int_{M_i}\cK_i\, \operatorname d\cH^{n-1}=0$ for every $i$, but $\int_{X}\cK\, \operatorname d\cH^{n-1}>0$. 
%However, if the limit space is smooth, then we have the following convergence result. 

\begin{Theorem B}\label{t:K.cont.smooth}
	Let $X_i\in\Alex^n(-1)$ and $n\ge 2$. If $(X_i,p_i)$ Gromov-Hausdorff converges to a smooth manifold $(M,p)$ without collapsing. If $\partial M=\varnothing$ or the doubling of $M$ is smooth, then  $\displaystyle\lim_{i\to\infty}\int_{B_1(p_i)} \cK_i\, \operatorname d\cH^{n-1}=0$ and $\displaystyle\lim_{i\to\infty}\int_{B_1(p_i)} \theta_i\, \operatorname d\cH^{n-2}=0$. 
\end{Theorem B}

We would like to point out that it is unknown if $X\setminus\cKS(X)$ is a smooth manifold away from a zero $\cH^{n-1}$-measure  subset. According to Theorem \ref{t:cont.sec}, this is the gap to improve the conclusion of Theorem B to the following convergence result
	 \begin{align}
		 	\displaystyle\lim_{i\to\infty}\int_{B_1(p_i)} {\operatorname R}_i\,\operatorname d \cH^n + \cK_i\,\operatorname d \cH^{n-1} +\theta_i\,\operatorname d \cH^{n-2}=\int_{B_1(p)} {\operatorname R}\,\operatorname d \cH^n.
		 \end{align}

This paper is organized as follows: In Section \ref{s:defn}, we define $\cK(x)$ and illustrate it with an example. In Section \ref{s:annuli}, we establish a new covering theorem for Alexandrov spaces. The theorem states that any compact set can be covered by a controlled number of convex annuli with centers on the given compact set (see Theorem \ref{t:cvx-cover}). We will apply this covering theorem recursively to form an annulus stratification of $X$. Furthermore, by establishing Lemma \ref{t:adj-angle}, we can adjust the intersection angles of the strata to be less than but close to $\frac\pi2$. In Section \ref{s:mono}, we establish the monotonic formula Corollary \ref{c:mono.corner}, which implies an upper bound of the integration of $\cK$ on each stratum (Lemma \ref{l:int.formula.K}), in terms of the derivatives of the volume of level sets. In Section \ref{s:sum.sing}, we prove Theorem A by integrating the monotonic formulas along with the stratification structures established in Section \ref{s:annuli}. We then prove Theorem B using Theorem A and the covering theorems in Section \ref{s:annuli}. In the Appendix, we give a simplified proof for Petrunin's Theorem \ref{t:int.scal}, using the covering techniques established in Section \ref{s:annuli}.

The author would like to thank Feng Wang and Jian Ge for fruitful discussions. 

\section{Definitions and Examples}\label{s:defn}

We define $\cK(x)$ in this section. Since we will integrate it with respect to the $(n-1)$-dimensional Hausdorff measure, it suffices to define it away from a set of co-dimension 2. 
%See Section \ref{s:sum.sing} for more details. 
Let $X\in\Alex^n(\kappa)$ without boundary. If $X$ has non-empty boundary, we will consider the doubling of $X$. A subset $W\subseteq X$ is said to be locally convex at a point $x$ if there exists $r>0$ so that $B_r(x)\cap W$ is convex. According to Theorem \ref{t:cvx-cover}, all but a finite number of points $x\in B_1$ admits a subset $W$ so that $x\in\partial W$ and $W$ is locally convex at $x$. The tangent cones $T_x(W)$ at these points are metric cones with boundaries. If $T_x(W)=\dR^{n-2}\times C([0,\beta])$, $0<\beta\le\pi$, we define the singularity angle $\Theta(x, \partial W)=\pi-\beta$, otherwise we let $\Theta(x, \partial W)=\pi$. Although $\Theta(x, \partial W)$ is determined by $W$ but not by $\partial W$ itself, we still use the latter in our notion. This is to highlight $x\in\partial W$ and that the singularity direction is normal to $\partial W$. We would like to point out that $\Theta(x,\partial W)$ only reflects the singularity of $W$ with respect to its intrinsic metric, but not the singularity of $X$. In particular, having $\Theta(x,W)>0$ does not imply $x\in\cS^{n-2}(X)$. For example, a triangle plane $W\subset\dR^2$ gives $\Theta(x,\partial W)>0$ at the vertex $x\in W$, but there are no singular points in $\dR^2$.

Let $W_t=\{x\in W\colon d(x,\partial W)\ge t\}$ be the sup-level set and $\partial W_t=\{x\in W_t\colon d(x,\partial W)= t\}$. Let $\phi_t(x)$ be the re-parameterized gradient flow (level set flow) of $d_{\partial W}|W$ so that $\phi_t(x)\in\partial W_t$. Since $W$ is locally convex at $x$, we have $\displaystyle\liminf_{t\to 0^+}T_{\phi_t(x)}(W_t)\ge T_x(W)$. That is, there is a distance non-increasing map from any Gromov-Hausdorff limit of $T_{\phi_t(x)}(W_t)$ as $t\to 0^+$ to $T_x(W)$. This implies
$$\displaystyle\limsup_{t\to 0^+}\Theta (\phi_t(x),\,\partial W_t)\le \Theta (x,\partial W).$$ 
In particular, if $\Theta (x,\partial W)=0$, then $\displaystyle\lim_{t\to 0^+}\Theta (\phi_t(x),\,\partial W_t)=0$. 
For $x\in\partial W$, we define the following directional singularity measure in the sense of distribution 
\begin{align}
	\cK(x, \partial W) &= \renewcommand{\arraystretch}{1.5}
	\left\{\begin{array}{@{}l@{\quad}l@{}}
		\displaystyle\limsup_{t\to 0^+}\frac{\Theta (\phi_t(x),\,\partial W_t)}{t}, 
		& \text{ if } \Theta(x,\partial W)=0,
		\\
		0, & \text{otherwise.}
	\end{array}\right.
	\label{s:sum.sing.e4}
\end{align}

%\red{If $n=2$, then $\cK(x, \partial W)$ is equivalent to the Gaussian curvature (or geodesic curvature) of $W$ at $x\in\partial W$}. 
If $\Theta(x,\partial W)=0$ and the metric in a neighborhood of $x$ is $C^2$, then $W$ can be chosen so that $\partial W_t$ are smooth for $t>0$ small. Then we have $\Theta (\phi_t(x),\,\partial W_t)\equiv 0$ and $\cK(x,\partial W)=0$. This shows that for points $x$ satisfying $\Theta (x,\,\partial W)>0$, the metric is not locally $C^2$. See Example \ref{e:2disk} for which $\cK(x,\partial W)$ is supported on a co-dimension 1 subset. 

%Let $X$ be a length metric space and $H\subset X$ be a hyper surface, whose intrinsic metric is smooth. Let $d_H\colon x\mapsto d(x, H)$ be the distance function to $H$. If $X$ is smooth in $B_r(x)$, where $x\in H$, then $d_H$ will be smooth at $x$. That is, the hyper surface $d_H^{-1}(t)\cap B_\epsilon(x)$ will be smooth in $(-\epsilon, \epsilon)$ for some $\epsilon>0$ small. This means that if $H$ is smooth at $x$ but its variation $d_H^{-1}(t)\cap B_r(x)$ is not smooth, then $x$ should be a singular point in $X$. 

To capture the singularities along all directions, we shall compute $\cK(x,\partial W^i)$ for $n$ almost orthogonal subsets $W^i$, where $x\in\partial W^i$ for every $i$. By Remark \ref{r:non-frame}, for any $\epsilon>0$, there is a co-dimension 2 subset $\mathcal A$, so that for any $x\notin\mathcal A$, there is a sequence of subsets $W^i$, $i=1,2,\cdots, n$ so that
\begin{enumerate}
	\renewcommand{\labelenumi}{(\roman{enumi})}
	\item $x\in\partial W^i$, 
	\item $W^i$ is locally convex at $x$. 
	\item $\{W^i\}$ are $\epsilon$-orthogonal.
\end{enumerate}
Such a sequence $\{W^i\}$ is called an $\epsilon$-frame (see Definition \ref{d:frame} for details). Let 
\begin{align}
	\cK(x)=\inf_{\{W^1, W^2, \dots, W^n\}}\sum_{i=1}^n\cK(x, \partial W^i),
	\label{d:K.e1}
\end{align}
where the infinitum is taking over all family of frames. 

Note that the value of $\sum_{i=1}^n\cK(x, \partial W^i)$ may depend on the choice of frames. However, in Example \ref{e:2disk}, as a measure, $\sum_{i=1}^n\cK(x, \partial W^i)\, \operatorname d\cH^{n-1}$ is equivalent over different choices of $\epsilon$-frames, provided $0<\epsilon<c(n)$ small. We conjecture that this is the case in general. That is, for $\epsilon$-frames $\{W_1^1, W_1^2, \dots, W_1^n\}$ and $\{W_2^1, W_2^2, \dots, W_2^n\}$ at $x$, it holds that
$$c_1(n)\sum_{i=1}^n\int_{B_1}\cK(x, \partial W^i_1)\, \operatorname d\cH^{n-1}
\le \sum_{i=1}^n\int_{B_1}\cK(x, \partial W^i_2)\, \operatorname d\cH^{n-1}
\le c_2(n)\sum_{i=1}^n\int_{B_1}\cK(x, \partial W^i_1)\, \operatorname d\cH^{n-1}.$$
If this is true, then the infinitum in (\ref{d:K.e1}) can be dropped. Example \ref{e:2disk} also indicates that $\cK(x)$ is equivalent to the Gaussian curvature at $x$ in dimension 2. 

As mentioned earlier, the following example illustrates the co-dimension 1 behavior of $\cK(x)$. It also shows that $\cK(x)$ correctly captures the $C^2$-singularities. 
\begin{example}\label{e:2disk}
	Let $\mathds D$ be a flat disk in $\dR^2$ with radius $r$. Let $X$ be the doubling of $\mathds D$. Let $Z\subset X$ be the glued two copies of $\partial \mathds D$. Construct $M_i$ by smoothing $X$ in a small neighborhood of $Z$, so that $M_i\overset{d_{GH}}\longrightarrow X$ as $i\to\infty$. Note that $X$ is flat away from the 1-dimensional subset $Z$. By the Gauss-Bonnet Theorem, 
	\begin{align}
		\int_{M_i}K_{g_i}\operatorname d vol_{g_i}=4\pi=\int_Z K_x \operatorname d s.
	\end{align} 
	It's easy to see that $K_x=\frac2r$ for every $x\in Z$. This shows that the curvature on $X$ is concentrated on the co-dimension 1 subset $Z$. The curvature measure $K_{g_i}\operatorname d vol_{g_i}$ converges to measure $K_x ds$, supported on $Z$. 
	
	We will show that $K_x=c\cdot\cK(x)$ and $Z=\cKS(X)$. For any $x\in Z$, there is a convex set $W$ so that $x\in\partial W$ and $B_{2\epsilon}(x)\cap\partial W$ is a local geodesic. Let $W_t$ be the sup-level set of the level set flow $\phi_t(x)$, which is defined in terms of $d_{\partial W}|_W$. Let $\delta(t)$ be the intersection angle of $\partial W_t$ and $Z$ at $\phi_t(x)$. Note that $\Theta(t)=\Theta(\phi_t(x), \partial W_t)$ is equal to the exterior angle of $W_t$ at $\phi_t(x)$. Direct computation shows that 
	\begin{align}
		\renewcommand{\arraystretch}{1.2}
		\left\{\begin{array}{@{}l@{\quad}l@{}}
			& r\cos\delta(t)=r\cos\delta(0)+t, 
			\\
			& \delta(t)=\delta(0)-\frac{\Theta(t)}{2}.
		\end{array}\right.
	\end{align}
%	\begin{align}
%	r\cos\left(\delta(0)-\frac{\Theta(t)}{2}\right)=r\cos\delta(0)+t.
%\end{align}
Then 
 	\begin{align}
 	\cK(x,\partial W_0)=\lim_{t\to 0^+}\frac{\Theta(t)}{t}= \frac2{r\sin\delta(0)}.
 \end{align}
It's clear that $\cK(y,\partial W_0)=0$ for any $y\in B_\epsilon(x)\cap\partial W\setminus\{x\}$. 
Note that $\delta(0)$ can be chosen arbitrarily close to $0$. Thus $\cK(x,\partial W)$ may approach to $\infty$. However, 
\begin{align}
	\int_0^T\cK(x,\partial W_t)dt=\int_0^T\frac2{r\sin\delta(t)}dt=\delta(0)-\delta(T)
\end{align}
is always bounded. One can also see (\ref{l:int.formula.K.e4}) for a generalized version. 

We now choose $W^1$ and $W^2$ so that their intersection angle is $\frac\pi2$. Let  $\delta_i$ be the intersection angle of $\partial W^i$ and $Z$ at $x$, $i=1,2$. We get 
	 \begin{align}
		\cK(x,\partial W^1)+\cK(x,\partial W^2)=\frac2r\left(\frac1{\sin\delta_1}+\frac1{\sin\delta_2}\right).
	\end{align}
Note that $\delta_1+\delta_2=\frac\pi2$. Choosing $\delta_1=\delta_2=\frac\pi4$ to minimize the above sum, we get 
$\cK(x)\le\frac{4\sqrt 2}{r}$.
\end{example}

\section{Covering by Convex Annuli}\label{s:annuli}

We develop a covering technique in this section. Each element in our covering will be an annulus-like region that satisfies a convex property. This technique will be applied recursively to create a stratified covering.

\subsection{Covering by good scale annulus}

Let $\in\Alex^n(-1)$ be an Alexandrov space. Let $p\in X$ and $A_a^b(p)=\{x\in X\colon a\le d(p,x)\le b\}$ denote the close annuli centered at $p$. In particular $A_0^r(p)=B_r(p)$ is the closed geodesic ball and $A_r^r(p)=\{x\in X\colon d(p,x)= r\}$ is the level set. 

\begin{definition}[Good scale annulus]
	An annulus region $A_a^b(p)$ is said to be an $(\epsilon, \sigma)$-good scale annulus if the following hold. 
	\begin{enumerate}
		\renewcommand{\labelenumi}{(\roman{enumi})}
		\item $a\le\sigma b$.
		\item There exists a metric cone $C_{p^*}(\Sigma)$ so that
		\begin{align}
			d_{GH}\Big(B_s(p), \, B_s(p^*)\Big)\le \epsilon s
			\label{defn:bad scale.e1}
		\end{align}
		for every $s\in[\sigma a,\sigma^{-1}b]$.
	\end{enumerate}
\end{definition}

\begin{theorem}[Covering by good scale annuli, Theorem 15.3.14 \cite{Li20}, Li-Naber \cite{LiNab20}]\label{t:ann_cover}
	Let $\epsilon, \sigma >0$ and $n\in\dN$. 
	For any $X\in\Alex^n(-1)$ and closed subset $\Omega\subseteq\bar B_1$, there exists a collection of $(\epsilon, \sigma)$-good scale annuli  $\mathcal C=\Big\{A_{a_i}^{b_i}(p_i)\Big\}$ so that  
	\begin{enumerate}
		\renewcommand{\labelenumi}{(\roman{enumi})}
		\item $\Omega\subseteq \cup_iA_{a_i}^{b_i}(p_i)$,
		\item $p_i\in\Omega$ for every $i$,
		\item $|\mathcal C|<C(n,\epsilon,\sigma)$.
	\end{enumerate}
\end{theorem}

\subsection{Covering by convex annuli}

We will show that a good scale annulus can be foliated by convex level sets. Consequently, the covering of good scale annuli can be adjusted to a covering of convex annuli. We first construct a $\left(-\frac{c(n)}{\epsilon}\right)$-concave function on $(\epsilon,\sigma)$-good scale annulus with $\sigma>0$ small. 

\begin{lemma}[Concave functions on good scale annulus]\label{l:cox-haull}
	There exists a constant $c(n)>0$ so that the following hold. Let $A_{a}^{b}(p)$ be an $(\epsilon, \sigma)$-good scale annulus with $\cH^{n}(B_{b}(p))\ge\nu b^n>0$. If $0<\sigma\le\epsilon^{2n}\nu^2$, then there exists a $\left(-\frac{c(n)}{\epsilon}\right)$-concave function $f$ on $B_{2b}(p)$ so that $f^{-1}(-s)\subseteq A_{(1-4\epsilon)s}^{(1+4\epsilon)s}(p)$ for any $0.5a\le s\le 2b$.
\end{lemma}
\begin{proof}
	Our constructions are similar to those in \cite{Kap02}, 3.6 in \cite{Per93a}, and 7.1.1 in \cite{Pet07}. However, we require more precise control of the constants. 
	
	Let $r=\sigma^{-1} b$. Define quadratic function
	$$\varphi(s)=(s-r)-\frac{\epsilon}{b}(s-r)^2.$$  %$\frac{\omega}{r}$ and $\omega=\epsilon^{-3}$
	For any unit speed geodesic $\gamma\subseteq B_{2b}(p)$, we have
	\begin{align}
		(\varphi(d(q,\gamma(s))))''
		&\le\frac1r-\frac{2\epsilon}{b}\left(\cos^2(\angle (\gamma^+(t), \uparrow_{\gamma(s)}^q))+\frac{d_q(\gamma(s))-r}{r}\right)
		\notag\\
		&\le \frac1r-\frac{2\epsilon}{b}\left(\cos^2(\angle (\gamma^+(t), \uparrow_{\gamma(s)}^q))-2\sigma \right).
		\label{l:cox-haull.e2}
	\end{align}
	
	%	We first show that the following estimates hold.
	%	\begin{enumerate}
		%		\renewcommand{\labelenumi}{(\roman{enumi})}
		%		\setlength{\itemsep}{1pt}
		%		\item If $x\in B_{2b}(p)$, then
		%		\begin{align}
			%			|\varphi\circ d_q(x)-(d_q(x)-r)|\le 4\epsilon b.
			%			\label{l:cox-ann-ei1}
			%		\end{align}
		%		\item If $x\in A_{0.5a}^{2b}(p)\setminus\{p\}$, then
		%		\begin{align} \big||\nabla_x(\varphi\circ d_q)|-1\big|\le 10\epsilon.
			%			\red{\text{This is not correct, but not needed.}}
			%			\label{l:cox-ann-ei2}
			%		\end{align}
		%		\item If $\gamma\subseteq B_{2b}(p)$ is a unit speed geodesic, then
		%		\begin{align}
			%			(\varphi(d(q,\gamma(s))))''
			%			&\le \frac1r-\frac{2\epsilon}{b}\left(\cos^2(\angle (\gamma^+(t), \uparrow_{\gamma(s)}^q))-2\epsilon^4\right).
			%			\label{l:cox-ann-ei3}
			%		\end{align}
		%	\end{enumerate}

	Let $Q=\{q_\alpha\}$ be a collection of maximum number of points in $A_r^r(p)$ so that $d(q_{\alpha_1}, q_{\alpha_2})\ge\delta r$ for any $\alpha_1\neq\alpha_2$. 
	%	Because $A_{a}^{b}(p)$ is an $(\epsilon,\sigma)$-good scale annuli, we have that
	%	\begin{align}
		%		|\mathcal Q|
		%		&\ge c(n)\frac{\cH^{n-1}(\partial B_{r}(p))}{(\delta r)^{n-1}}
		%		\ge c(n)\frac{\cH^{n-1}(\partial B_{b}(p))}{\delta^{n-1} b^{n-1}}
		%		\notag \\
		%		&\ge c(n)\frac{\cH^{n}(B_{b}(p))-(\sigma b)^n}{\delta^{n-1} b^{n}}
		%		\notag \\
		%		&\ge c(n)\frac{\nu b^n-(\sigma b)^n}{\delta^{n-1} b^{n}}
		%		\notag \\
		%		&\ge \frac{c(n)\nu}{\delta^{n-1}},
		%		\label{l:cox-haull.e4}
		%	\end{align}
	%	for $0<\sigma<\frac12\nu$.
	For each $\alpha$, let $\mathcal Q^\alpha=\{q^\alpha_\beta\}$ be a collection of maximum number of points in $B_{\delta r}(q_\alpha)\cap A_r^r(p)$ so that $d(q^\alpha_{\beta_1}, q^\alpha_{\beta_2})\ge\delta_1\delta r$ for any $\beta_1\neq\beta_2$. Because $A_{a}^{b}(p)$ is an $(\epsilon,\sigma)$-good scale annulus, we have that
	\begin{align}
		|\mathcal Q^\alpha|
		&\ge c(n)\frac{\cH^{n-1}(\partial B_{r}(p)\cap B_{\delta r}(q_\alpha))}{(\delta_1\delta r)^{n-1}}
		\ge c(n)\frac{\cH^{n-1}(\partial B_{r}(p))}{(\delta_1\delta r)^{n-1}}\cdot \left(\frac{\delta r}{r}\right)^{n-1}
		\notag\\
		&\ge c(n)\frac{\cH^{n-1}(\partial B_{b}(p))}{\delta_1^{n-1} b^{n-1}}
		\ge c(n)\frac{\cH^{n}(B_{b}(p))-(\sigma b)^n}{\delta_1^{n-1} b^{n}}
		\notag \\
		&\ge c(n)\frac{\nu b^n-(\sigma b)^n}{\delta_1^{n-1} b^{n}}
		\notag \\
		&\ge \frac{c(n)\nu}{\delta_1^{n-1}},
		\label{l:cox-haull.e4}
	\end{align}
	for $0<\sigma<\frac12\nu$.

	Define $$\dsp f_\alpha=\frac{1}{|\mathcal Q^\alpha|}\sum_{q^\alpha_\beta\in\mathcal Q^\alpha}\varphi\circ d_{q^\alpha_\beta}.$$ 
	Now we estimate the concavity of $f_\alpha$. For any $q\in B_r(p)$, by (\ref{l:cox-haull.e2}) and $b=\sigma r$, we have
	\begin{align}
		(\varphi(d(q,\gamma(s))))''
		&\le \frac1r+\frac{4\epsilon\sigma}b=\frac1r+\frac{4\epsilon}{r}\le \frac2r.
		\label{l:cox-haull.e5}
	\end{align}
    Given a direction $\xi\in\Sigma_x$, let $\mathcal Q^\alpha_x(\xi)=\left\{q^\alpha_\beta\in \mathcal Q^\alpha\colon \left|\angle \left(\xi, \uparrow_x^{q^\alpha_\beta}\right)-\frac\pi2\right|\le \frac1{10}\eta\right\}$. 
	If $q\in\mathcal Q^\alpha\setminus \mathcal Q^\alpha_{\gamma(s)}(\gamma^+(s))$, we have $|\cos(\angle (\gamma^+(t), \uparrow_{\gamma(s)}^q))|\ge\eta$. Then (\ref{l:cox-haull.e2}) gives
	\begin{align}
		(\varphi(d(q,\gamma(s))))''
		&\le \frac1r-\frac{2\epsilon}{b}(\eta^2-2\sigma)\le\frac{2}{r}-\frac{4\epsilon\eta^2}{\sigma r}
		\le -\frac{\epsilon\eta^2}{\sigma r},
		\label{l:cox-haull.e6}
	\end{align}
	provided $\sigma\le\epsilon\eta^2$. 
	
	Note that $\mathcal \mathcal Q^\alpha_x(\xi)$ is the minority in $\mathcal Q^\alpha$. In fact, $d(p,x)\le 2b =2\sigma r$ and we have 
	$$\ang{x}{q^\alpha_{\beta_i}}{q^\alpha_{\beta_j}}
	\ge \tang{x}{q^\alpha_{\beta_i}}{q^\alpha_{\beta_j}}
	\ge \tang{p}{q^\alpha_{\beta_i}}{q^\alpha_{\beta_j}}-10\sigma
	\ge \frac12\delta\delta_1-10\sigma\ge\frac14\delta\delta_1,$$
	provided $\sigma\le\frac1{50}\delta\delta_1$. 
	Thus $\displaystyle |\mathcal Q^\alpha_x(\xi)|\le \frac{c(n)\eta}{(\delta\delta_1)^{n-1}}$.
	By (\ref{l:cox-haull.e4}), (\ref{l:cox-haull.e5}) and (\ref{l:cox-haull.e6}), we have
	\begin{align}
		f_\alpha(\gamma(s))''
		&= \frac{1}{|\mathcal Q^\alpha|}\sum_{q^\alpha_\beta\in\mathcal Q^\alpha}(\varphi(d(q^\alpha_\beta,\gamma(s))))''
		\notag\\
		&\le \frac{1}{|\mathcal Q^\alpha|}\left(\sum_{q^\alpha_\beta\in \mathcal Q^\alpha_{\gamma(s)}(\gamma^+(s))}(\varphi(d(q^\alpha_\beta,\gamma(s))))''
		+ \sum_{q^\alpha_\beta\notin \mathcal Q^\alpha_{\gamma(s)}(\gamma^+(s))}(\varphi(d(q^\alpha_\beta,\gamma(s))))''\right)
		\notag\\
		&\le \frac{c(n)\delta_1^{n-1}}{\nu}\left(\frac{\eta}{(\delta\delta_1)^{n-1}}\cdot\frac2r 
		- \left(\frac{\nu}{\delta_1^{n-1}}-\frac{\eta}{(\delta\delta_1)^{n-1}}\right)\cdot\frac{\epsilon\eta^2}{\sigma r}\right)
		\notag \\
		&\le \frac{c(n)\eta}{\nu r}\cdot\left(\frac{2}{\delta^{n-1}}-\frac{\epsilon\eta\nu}{\sigma}+\frac{\epsilon\eta^2}{\delta^{n-1}\sigma}\right)
		\notag\\
		&\le \frac{c(n)\eta}{\nu r}\cdot\left(\frac{2\epsilon\eta^2}{\delta^{n-1}\sigma}-\frac{\epsilon\eta\nu}{\sigma}\right)
		\notag\\
		&= \frac{c(n)\epsilon\eta^2}{\nu r\sigma}\cdot\left(\frac{2\eta}{\delta^{n-1}}-\nu\right)
		\notag\\
		&\le -\frac{c(n)}{\epsilon},
		\label{l:cox-haull.e7}
	\end{align}
	provided $\sigma\le\epsilon^2\eta^2$ and $\eta\le\frac{1}{4}\delta^{n-1}\nu$. For example, we can choose $\delta=\delta_1=4\epsilon$, $\eta=\epsilon^{n-1}\nu$ and $\sigma=\epsilon^2\eta^2=\epsilon^{2n}\nu^2$. In particular, for any $\lambda<0$, function $f_\alpha$ is $\lambda$-concave, provided $0<\epsilon\le\frac{c(n)}{|\lambda|}$.
	
	Now define $\displaystyle f(x)=\min_\alpha f_\alpha(x)$. It's clear that $f$ is also $\left(-\frac{c(n)}{\epsilon}\right)$-concave in $B_{2b}(p)$. We claim that $|\nabla_x f|$ is close to 1 for any $x\in A_{0.5a}^{2b}(p)$. 
	For any $q\in \partial B_{r}(p)$ and $x\in B_{2b}(p)$, we have
	\begin{align}
		|\varphi\circ d_q(x)-(d_q(x)-r)|
		&=\frac\epsilon{b}(d_q(x)-r)^2\le \frac\epsilon{b} d^2_p(x)\le 2\epsilon \cdot d_p(x). 
		\label{l:cox-haull.e1}
	\end{align}
	On one hand, for any $q_\alpha\in \mathcal Q$ and $q^\alpha_\beta\in \mathcal Q^\alpha$, we have 
	\begin{align}
		\varphi\circ d_{q^\alpha_\beta}(x)\ge d_{q^\alpha_\beta}(x)-r -2\epsilon \cdot d_p(x)\ge -d_p(x)-2\epsilon \cdot d_p(x)\ge -(1+2\epsilon) \cdot d_p(x)
		\label{l:cox-haull.e8}
	\end{align}
	Thus $f_\alpha(x)\ge -(1+2\epsilon) \cdot d_p(x)$ for every $\alpha$ and 
	\begin{align}
		f(x)\ge -(1+2\epsilon) \cdot d_p(x).
		\label{l:cox-haull.e9}
	\end{align}
	On the other hand, for any $x\in A_{0.5a}^{2b}(p)$, by the good scale annulus structure, there exists $q_\alpha\in \mathcal Q$ so that $\ang{p}{q_\alpha}{x}\le 10(\delta+\epsilon)$. Moreover, 
	\begin{align}
		\ang{p}{q_\alpha}{x}\le 10(\delta+\epsilon+\delta\delta_1)\le 60\epsilon,
		\label{l:cox-haull.e10}
	\end{align} 
	for every $q^\alpha_\beta\in\mathcal Q^\alpha$, provided $0<\delta\le 4\epsilon$. By triangle comparison, this implies
	\begin{align}
		d_{q^\alpha_\beta}(x)\le r-d_p(x)\cos(60\epsilon).
		\label{l:cox-haull.e11}
	\end{align}
	Thus by (\ref{l:cox-haull.e1}), we have
	\begin{align}
		\varphi\circ d_{q^\alpha_\beta}(x)
		&\le d_{q^\alpha_\beta}(x)-r +2\epsilon \cdot d_p(x)\le -(\cos(60\epsilon)-2\epsilon)d_p(x)
		\notag\\
		&\le -(1-4\epsilon)d_p(x).
		\label{l:cox-haull.e12}
	\end{align}
	Therefore, 	
	\begin{align}
		f(x)\le f_\alpha(x)\le -(1-4\epsilon)d_p(x).
		\label{l:cox-haull.e13}
	\end{align}
	Combine (\ref{l:cox-haull.e9}) and (\ref{l:cox-haull.e13}), we have 
	\begin{align}
		\left|\frac{f(x)+d_p(x)}{d_p(x)}\right|\le 4\epsilon.
		\label{l:cox-haull.e14}
	\end{align}
	In particular, $f^{-1}(-s)\subseteq A_{(1-4\epsilon)s}^{(1+4\epsilon)s}(p)$ for any $s\le 2b$.
\end{proof}

Now we show that the above concave function induces a convex annulus. 

\begin{definition}[Convex annulus]\label{d:cvx.ann}
	A triple $(W, p, T)$ is said to be an $(\epsilon,\lambda)$-convex annulus if there exist $b>a\ge 0$ so that the following hold. 
	\begin{enumerate}
		\renewcommand{\labelenumi}{(\roman{enumi})}
		\item (Good scale) $A_{a}^{b}(p)$ is an $(\epsilon,\epsilon)$-good scale annulus and $A_{a}^{b}(p)\subseteq W\subseteq A_{0.5a}^{2b}(p)$.		
		\item (Convexity) There exists a hyper surface $H\subset X$ so that $W=\cup_{0\le t\le T} H_t$, where $H_t= \{x\in W\colon d(x,H)=t\}$ and the distance function $d_{H_t}$ is $\lambda$-concave on the sup-level set $W_t=\cup_{t\le s< T} H_s$.  
		%\item The sup-level set $\cup_{s\ge t} W_s$ is convex for every $t_1\le t\le t_2$.
		%\item There is a unique (maximum) point $x_0\in X$ so that $d(x_0,W_0)=\sup_{f(x)\ge 0}(d(x, W_0))$.
	\end{enumerate}
\end{definition}
As an abuse of notation, for a convex annulus $(W,p, T)$, we denote the associated $H$ and $H_t$ by $\partial W$ and $\partial W_t$ respectively.  We may also omit $p$ or $T$ in $(W,p, T)$ if it there is no ambiguity.

\begin{remark}
	In our definition, the boundary $\partial W$ of a convex annulus $W$ is a convex level set. However, $W$ may not be a convex set if $a\neq 0$ in the associated annular region $A_{a}^{b}(p)$.
\end{remark}

\begin{lemma}[Convex annulus on good scale annulus]\label{l:cox-ann} There exists a constant $c(n)>0$ so that the following hold. Let $A_{a}^{b}(p)$ be an $(\epsilon, \sigma)$-good scale annulus with $\cH^{n}(B_{b}(p))\ge\nu b^n>0$. If $0<\sigma\le\epsilon^{2n}\nu^2$, then there exists an $\left(\epsilon, -\frac{c(n)}{\epsilon}\right)$-convex annulus $W$ so that $A_{a}^{b}(p)\subseteq W\subseteq A_{0.5a}^{2b}(p)$.
\end{lemma}

\begin{proof}
	 Let $f$ be defined as in Lemma \ref{l:cox-haull} and $H=f^{-1}(-1.8b)$. Then $H\subseteq A_{(1-4\epsilon)1.8b}^{(1+4\epsilon)1.8b}(p)\subseteq A_{1.5b}^{2b}(p)$ and 
	 \begin{align}
	 	d_{GH}(H, A_{1.8b}^{1.8b}(p))<10\epsilon b.
	 	\label{l:cox-ann.e1}
	 \end{align} 
	To construct the convex annulus $W$, we let 
	$$H_t=\big\{x\in B_{2b}(p)\colon f(x)\ge -1.8b\text{ and } d(x,H)=t\big\}$$
	for $0< t< T=1.8b-0.8a$. Define $W_t=\cup_{t\le s< T} H_t$ and $W=W_0$. By the good scale condition on $A_a^b(p)$, we have $A_{a}^{b}(p)\subseteq W=\cup_{0\le t< T} H_t\subseteq A_{0.5a}^{2b}(p)$. Thus $W$ satisfies Definition \ref{d:cvx.ann} (i).

	It follows from the good scale and (\ref{l:cox-ann.e1}) that $|\nabla_xd_H-1|<20\epsilon$ for any $x\in A_{0.5a}^{2b}(p)\setminus \{p\}$. Therefore, we have $H_t\subseteq A_{(1-30\epsilon)(1.8b-t)}^{(1+30\epsilon)(1.8b-t)}(p)$,  $d_{GH}(H, A_{1.8b-t}^{1.8b-t}(p))<60\epsilon b$ and $||\nabla_xd_{H_t}|-1|<60\epsilon$ for any $x\in A_{0.5a}^{2b}(p)\setminus \{p\}$.
	
	Let $\lambda=-\frac{c(n)}{\epsilon}<-100$. Direct computation shows that level sets $H_t$, $0\le t< T$ are all ``very" convex and $d_{H_t}$ is $\frac12\lambda$-concave on $W_t$ for $\epsilon>0$ small. For details of the computation one can refer to \cite{AB10}. Let us outline the proof below. By (\ref{l:cox-haull.e14}), we have $||\nabla_xf|-1|\le 4\epsilon$ for any $x\in A_{0.5a}^{2b}(p)\setminus \{p\}$. By Theorem 1.6 \cite{AB10}, we have that the extrinsic curvature of $H=f^{-1}(-1.8b)$ is at least $|\lambda|$, by choosing $\epsilon>0$ small.
	
	By Corollary 1.9 \cite{AB10} and $X\in\Alex^n(-1)$, we have that $d_{H}|_{B_{1.5b}(p)}$ is concave. Moreover, because the extrinsic curvature of level sets are not less than the model case, due to (4.4) in \cite{AB10}, we have that $d_{H}$ is $\frac12\lambda$-concave on $W$. By Theorem 1.6 \cite{AB10} again, we have that the extrinsic curvature of $H_t$ is at least $|\lambda|$ for every $0\le t< T$. Apply Corollary 1.9 \cite{AB10} once more. We get that $d_{H_t}$ is $\frac14\lambda$-concave on $W_t$ for every $0\le t< T$.
\end{proof}

\begin{remark}
	Because $\partial W=H$ is very convex, the gradient flow of $d_{\partial W}|_{B_{1.5b}(p)}$ will contract $W\supseteq A_a^b(p)$ to a point. However, we don't know whether it contracts to a point in finite time if $a>0$. This is because $\nabla_xd_H$ may not be close to 1 in $B_s(p)$ for $0<s\le a$ small. By the same reason, the concavity of $d_{H_t}$ may not be preserved for $t>T$. 
\end{remark}

%Note that every point $p\in X$ admits a convex neighborhood. Thus for any compact set in $X$, there is a finite cover by convex sets. However, the cardinality of the covering can't be bounded by the dimension. (for example...) If we relax the convexity and cover the compact set by the above convex annuli, we will be able to control the cardinality of the covering.

\begin{theorem}[Covering by convex annuli]\label{t:cvx-cover} 
	Let $(X,p)\in \Alex^n(-1)$ and $\Vol{B_1(p)}\ge \nu>0$. 
	For any closed subset $\Omega\subseteq\bar B_1(p)$, there exists a collection of $\left(\epsilon, -\frac{c(n)}{\epsilon}\right)$-convex annuli  $\cW=\{(W_i, p_i, T_i)\}$ so that  
	\begin{enumerate}
	\renewcommand{\labelenumi}{(\roman{enumi})}
	\item $\Omega\subseteq \cup_iW_i$,
	\item $p_i\in\Omega$ for every $i$,
	\item $|\cW|<C(n,\nu,\epsilon)$.
\end{enumerate}
\end{theorem}

\begin{proof}
	By Theorem \ref{t:ann_cover}, we can cover $\Omega$ by a collection of $(\epsilon,\sigma)$-good scale annuli $\mathcal C=\Big\{A_{a_i}^{b_i}(p_i)\Big\}$, with $p_i\in\Omega$ and $|\mathcal C|<c(n,\epsilon,\sigma)$. For each $i$, by Bishop-Gromov volume comparison, we have
	$$\frac{\cH^{n}(B_{b_i}(p))}{b_i^n}\ge c(n)\cdot\cH^{n}(B_{1}(p))\ge c(n)\nu>0.
	$$
	By choosing $\sigma=\epsilon^{2n}\nu^2$ as in the assumption of Lemma \ref{l:cox-ann}, every good scale annulus $A_{a_i}^{b_i}(p_i)$ can be covered by an $\left(\epsilon, -\frac{c(n)}{\epsilon}\right)$-convex annulus $W_i$ with center at $p_i$.
\end{proof}

\subsection{Adjust the intersection angles}

In order to cover $B_1$ and capture the singularities in all directions, we will recursively apply Theorem \ref{t:cvx-cover} to establish a stratified covering.  Comparison geometry will be applied on each of the strata in the proof of the monotonic formula in Section \ref{s:mono}. This requires that the intersection angles between convex annuli are not greater than $\frac\pi2$. 

\begin{definition}\label{d:frame}
	A sequence of $(\epsilon,\lambda)$-convex annuli $(W^i, p_i, T_i)$, $i=1,2,\dots,m$ is called an $(\epsilon,\lambda, \delta)$-frame if 
	\begin{align}
		-2\delta \le \langle \nabla_xd_{\partial W^i_t}, \nabla_xd_{\partial W^j_s}\rangle\le -\delta
	\end{align}
for any $x\in \partial W^i_t\cap \partial W^j_s$, $0\le t< T_i$ and $0\le s< T_j$. The intersection $\displaystyle M=\cap_{i=1}^m\partial W_{t_i}^i$ is called an $(\epsilon,\lambda,\delta)$-corner space of $\{W^1_{t_1}, W^2_{t_m},\dots,W^m_{t_m}\}$. It's clear that $\dim(M)=n-m$. For convenience, we define $M=X$ if $m=0$. 
\end{definition}

\begin{lemma}[Adjust the intersection angles]\label{t:adj-angle}
	Let $(W^i, p_i, T_i)$, $i=1,2,\dots,m$ be an $(\epsilon,\lambda, \delta)$-frame. Let $\displaystyle M=\cap_{i=1}^m\partial W^i$ be an $(\epsilon,\lambda,\delta)$-corner space. The following holds for $0<10\epsilon<\delta<\frac1{200n}$ small and $\lambda<0$. 
	
	For any $(\epsilon,\lambda)$-convex annulus $(W,p,T)$ with $p\in \cap \partial W^i$, given by concave function $f(x)$ in Lemma \ref{l:cox-ann}, there exists a $(2\epsilon,\frac12\lambda)$-convex annulus $(\widetilde W,p, \frac12T)$ so that $W_{\frac12T}\subseteq \widetilde W\subseteq W_{2T}$ and 
	\begin{align}
		-24\delta \le \langle \nabla_xd_{\partial W^i}, \nabla_xd_{\partial \widetilde W_{t}}\rangle\le -12\delta
		\label{t:adj-angle.e1}
	\end{align}
	for any $1\le i\le m$, $x\in M\cap \partial \widetilde W_t$ and $0\le t<\frac12T$. In particular 
	$$\{(W^1, p_1, T_1), (W^2, p_2, T_2), \dots, (W^m, p_m, T_m), (\widetilde W,p,\frac12T)\}$$ 
	is an $(\epsilon,\frac12\lambda, 12\delta)$-frame.
\end{lemma}

\begin{proof}
	For $x\in M\cap W$, define
	\begin{align}
		\tilde f(x)=f(x)+20\delta\cdot\sum_{i=1}^m \left(d(x,p_i)-d(p, p_i)\right).
		\label{t:adj-angle.e2}
	\end{align}
	It's clear that $\tilde f$ is $\frac12\lambda$-concave on the good-scale annulus for $\epsilon>0$ small. The corresponding convex annulus $\widetilde W$ is constructed as in Lemma \ref{l:cox-ann}. We have $W_{\frac12T}\subseteq \widetilde W\subseteq W_{2T}$ because
	\begin{align}
		\left|\frac{\tilde f(x)-f(x)}{d(x,p)}\right|\le 20\delta.
		\label{t:adj-angle.e3}
	\end{align}
Because $(W,p)$ is an $(\epsilon,\lambda)$-convex annulus and $p, x\in M$, we have 
\begin{align}
	|\langle \nabla_xd_{\partial W^i}, \nabla_xf\rangle|\le 10\epsilon.
	\label{t:adj-angle.e4}
\end{align}
    Because $(W^i,p_i)$ are $(\epsilon,\lambda)$-convex annuli, we have
    \begin{align}
    	|\langle \nabla_xd_{\partial W^i}, \nabla_xd_{p_i}\rangle+1|\le 10\epsilon.
    	\label{t:adj-angle.e5}
    \end{align}
    Because $\{(W^i,p_i)\}$ is an $(\epsilon,\lambda, \delta)$-frame, we have
    \begin{align}
    	|\langle \nabla_xd_{\partial W^i}, \nabla_xd_{p_j}\rangle|\le 2\delta+10\epsilon.
    	\label{t:adj-angle.e6}
    \end{align}
Combing (\ref{t:adj-angle.e4}) -- (\ref{t:adj-angle.e6}), we have
	\begin{align}
	\langle \nabla_xd_{\partial W^i}, \nabla_x\tilde f\rangle
	&\le 10\epsilon+20\delta(-1+10\epsilon+m(2\delta+10\epsilon))
	\notag\\
	&\le-18\delta+30\epsilon\le -15\delta,
\end{align}
provided $10\epsilon<\delta<\frac1{200n}$. Similarly, 
	\begin{align}
	\langle \nabla_xd_{\partial W^i}, \nabla_x\tilde f\rangle
	&\ge -10\epsilon-20\delta(1-10\epsilon+m(2\delta+10\epsilon))
	\notag\\
	&\ge-21\delta-11\epsilon\ge -23\delta.
\end{align}
Inequality (\ref{t:adj-angle.e1}) follows by the $\epsilon$-good scale on $\widetilde W$ and
\begin{align}
	|\langle \nabla_xd_{\partial \widetilde W}, \nabla_x\tilde f\rangle-1|\le 10\epsilon<\delta
\end{align}
for every $x\in\partial \widetilde W$. 
\end{proof}

We now recursively apply Lemma \ref{t:cvx-cover} to cover corner spaces and adjust the intersections angles by Lemma \ref{t:adj-angle}. Then we have the following stratified covering theorem. 

\begin{lemma}[Stratified covering by convex annuli]\label{t:cov-angle} There exists $c(n)>0$ so that the following holds for $\lambda<0<\epsilon<\frac{c(n)}{|\lambda|}$ and 	$0<10\epsilon<\delta<\frac1{200n}$.  
	Let $(W^i, p_i, T_i)$, $i=1,2,\dots,m$ be an $(\epsilon,\lambda, \delta)$-frame and $\displaystyle M=\cap_{i=1}^m\partial W^i$ be an $(\epsilon,\lambda,\delta)$-corner space. Then there exists an $(\epsilon,\frac12\lambda)$-convex annulus covering $\{(W^{m+1,\ell}, p'_\ell, T'_\ell)\}_{\ell=1}^N$ of $M\cap B_1$, so that 
	$$\{(W^1, p_1, T_1), (W^2, p_2, T_2), \dots, (W^m, p_m, T_m), (W^{m+1,\ell}, p'_\ell, T'_\ell)\}$$ 
	is an $(\epsilon,\frac12\lambda, 12\delta)$-frame for any $1\le\ell\le N=c(n,\epsilon, \lambda,\nu)$.
\end{lemma}

\begin{remark}\label{r:covex-cov}
	By choosing $\lambda_{i+1}=\frac12\lambda_i$, $\lambda_1=-\frac{2^n}\delta$, $\delta_{i+1}=12\delta_i$, $\delta_1=\frac\delta{200n\cdot 12^n}$ and $\epsilon=\min\{\frac{c(n)}{|\lambda_1|}, \frac1{10}\delta_1\}$, the above lemma can be applied on each recursively constructed $(\epsilon,\lambda_i, \delta_i)$-frame and $(\epsilon,\lambda_i, \delta_i)$-corner space for $i=1,2,\dots,n-1$. In this construction, we have $\lambda_i\le -\delta^{-1}$ and $\epsilon<\delta_i<\delta$ for every $i$. When we apply Lemma \ref{t:cov-angle}, we will assume this for every frame and corner space, unless otherwise stated. As a convention, an $(\epsilon,\lambda,\delta)$-corner space will be called a $\delta_0$-corner space, and an $(\epsilon,\lambda, \delta)$-frame will be called a $\delta_0$-frame, provided $\lambda=-c_1(n)\delta_0^{-1}$, $\epsilon=c_2(n)\delta_0$ and $\delta=c_3(n)\delta_0$.  Similarly, an $(\epsilon,\lambda)$-convex annulus will be called an $\epsilon_0$-convex annulus for $\lambda=-c_1(n)\epsilon_0^{-1}$ and $\epsilon=c_2(n)\epsilon_0$.
\end{remark}

\begin{remark}\label{r:non-frame}
	By the above covering theorem, on each corner space $M$, there are at most $C(n,\nu,\epsilon)$ points which do not admit a frame of next stratum. These points are the centers of the convex annuli in the covering of $M$. Therefore, away from a set of co-dimension 2, every point on $X$ admit an $\epsilon$-frame. 
\end{remark}

\section{Monotonic Formula on Intersection of Level Sets}\label{s:mono}
In this section, we establish a monotonic formula (Corollary \ref{c:mono.corner}). It gives a bound for the integration of $\Theta$ on the corner spaces, in terms of the derivatives of the volume of level sets. 

\subsection{Generalized Minkowski content}

%Let us begin with need a revised formula of Minkowski content. 

\begin{definition} Let us list some definitions of rectifiable sets in this paper. 
	\begin{enumerate}
		\renewcommand{\labelenumi}{(\roman{enumi})}
		\setlength{\itemsep}{1pt}
		\item A set $E$ is said to be $m$-rectifiable if there exists a bounded set $\tilde E\subset\dR^m$ and a Lipschitz onto function $\Psi\colon \tilde E\to E$. 
		\item A set $E$ is said to be countably $m$-rectifiable if $E$ is a  union of countable collection of $m$-rectifiable sets. 
		\item A set $E$ is said to be $\cH^m$-rectifiable, if there exists a countably $m$-rectifiable set $F$ so that the Hausdorff measure $\cH^m(E\setminus F)=0$. 	
	\end{enumerate}
\end{definition}

\begin{definition}[p. 273 \cite{Fed69}]
  Let $0\le m\le n$ be integers and $E\subset \dR^n$ be a subset. The $m$-dimensional upper Minkowski content is
  \begin{align}
    {\mathscr M}^{*m}(E)=\limsup_{r\to 0^+}\frac{\cH^n(B_r(E))}{\omega(n-m)r^{n-m}},
    \label{s:mono.e1}
  \end{align}
  and the $m$-dimensional lower Minkowski content is
  \begin{align}
    {\mathscr M}_*^{m}(E)=\liminf_{r\to 0^+}\frac{\cH^n(B_r(E))}{\omega(n-m)r^{n-m}},
    \label{s:mono.e2}
  \end{align}
where $\omega(k)$ is the volume of unit ball in $\dR^k$. 
  If ${\mathscr M}^{*m}(E)={\mathscr M}_*^{m}(E)$, then the common value $\mathscr M^m(E)$ is called the $m$-dimensional Minkowski content.
\end{definition}

It is known that if $E$ is a closed $m$-rectifiable set, then $\mathscr M^m(E)=\cH^m(E)$ (Theorem 3.2.39 \cite{Fed69}). By definition, if $E$ is $\cH^m$-rectifiable, then ${\mathscr M}^{*m}(E)\ge\cH^m(E)$. We will need the following generalization for these results.

Let $f\colon E\to \dR_{\ge 0}$ be a bounded function. 

\begin{definition}
  Let $0\le m\le n$ be integers and $E\subset \dR^n$ be a subset. The $m$-dimensional upper $f$-Minkowski content is
  \begin{align}
    {\mathscr M}_f^{*m}(E)=\limsup_{r\to 0^+}\frac{\cH^n\left(\cup_{x\in E}B_{f(x)r}(x)\right)}{\omega(n-m)r^{n-m}},
    \label{s:mono.e3}
  \end{align}
  and the $m$-dimensional lower $f$-Minkowski content is
  \begin{align}
    {\mathscr M}_{*f}^{m}(E)=\liminf_{r\to 0^+}\frac{\cH^n\left(\cup_{x\in E}B_{f(x)r}(x)\right)}{\omega(n-m)r^{n-m}},
    \label{s:mono.e4}
  \end{align}
  If ${\mathscr M}_f^{*m}(E)={\mathscr M}_{*f}^{m}(E)$, then the common value is called the $m$-dimensional $f$-Minkowski content $\mathscr M_f^m(E)$.
\end{definition}
By the Caratheodory's construction in Section 2.10 \cite{Fed69}, we can substitute the estimation function $\zeta(S)=\omega(k)2^{-m}(\diam(U))^m$ by $\zeta_f(S)=\omega(m)f(x)2^{-m}(\diam(U))^m$, where $x\in U$, and  obtain the corresponding $f$-measures. In particular, we have the $m$-dimensional $f$-Hausdorff measure
\begin{align}
  \cH_f^m(E)
  &=\lim_{\delta\to 0^+}\inf\left\{\sum_{i=1}^\infty{\omega(k)f(x_i)2^{-m}(\diam(U)_i)^m}\colon x\in U_i, \, E\subseteq\bigcup_{i=1}^\infty U_i, \, \diam(U_i)<\delta \right\}
  \notag\\
  &=\int_E f(x)\operatorname d\cH^m.
  \label{s:mono.e5}
\end{align}
%Because $f$ is bounded, we have that if $E$ is $\cH^m$-rectifiable, then it is also $\cH_f^m$-rectifiable.
By substituting Lebegsue measure $\mathscr L^m(E)$ by the $f$-Lebegsue measure $\mathscr L_f^m(E)=\int_E f(x) d\mathscr L^m$, we obtain the $m$-dimensional $f$-Gross measure (\cite{Fed69}, 2.10.4)
\begin{align}
  \mathscr{G}_f^m(E)
  &=\lim_{\delta\to 0^+}\inf\left\{\sum_{i=1}^\infty\sup_{\mathfrak p\in\mathbf O^*(m,k)}\Big\{\mathscr L_f^m(\mathfrak p(U_i))\Big\}\colon  E\subseteq\bigcup_{i=1}^\infty U_i, \, \diam(U_i)<\delta \right\},
  \label{s:mono.e6}
\end{align}
where $\mathbf O^*(m,k)$ is the set of orthogonal projections $\mathfrak p\colon \dR^n\to\dR^m$ (\cite{Fed69}, 1.7.4) and $U_i$ are all Borel subsets.

Now let $u\colon \dR^n\to \dR_{\ge 0}$ be a bounded function. By the same proof as 3.2.37 \cite{Fed69} and Fubini's Theorem, we have
\begin{align}
  {\mathscr M}_{*u}^m(E)\ge \mathscr{G}_{u^{n-m}}^m(E).
  \label{s:mono.e7}
\end{align}
If $E$ is $\cH^m$-rectifiable, then by the same argument as in 3.2.26 \cite{Fed69}, we have  $$\mathscr{G}_f^m(E)=\cH_f^m(E).$$
Apply this to $f=u^{n-m}$. We have
\begin{align}
  {\mathscr M}_{*u}^m(E)\ge \mathscr{G}_{u^{n-m}}^m(E)=\cH_{u^{n-m}}^m(E)=\int_Eu^{n-m}(x)\operatorname d\cH^m.
  \label{s:mono.e8}
\end{align}
By a similar argument as in 3.2.38 and 3.2.39 \cite{Fed69}, one can prove
\begin{align}
	\mathscr M_u^m(E)=\int_Eu^{n-m}(x)\operatorname d\cH^m,
	\label{s:mono.e10}
\end{align}
if $E$ is $k$-rectifiable.
%Though (\ref{s:mono.e8}) is enough for our purpose, b

Since a finite dimensional Alexandrov space can be isometrically embedded into a finite dimensional Euclidean space, the above results also hold if $E$ is a subset of Alexandrov space. In our application, $m=k-2$ and the balls are intersected with a corner space $M$ of dimension $n=k-1$. That is, $n-m=1$ and
\begin{align}
	\liminf_{r\to 0^+}\frac{\cH^{k-1}\left(\cup_{x\in E}B_{u(x)r}(x,\Omega)\right)}{\omega(1)r}
	={\mathscr M}_{*u}^{k-2}(E)\ge \cH_{u}^{k-2}(E)=\int_Eu(x)\operatorname d\cH^{k-2}.
	\label{s:mono.e9}
\end{align}

%\begin{theorem}\label{t:pm-cont}
%  Let $E\subset\dR^n$ be a closed set and $u\colon E\to \dR_{\ge 0}$ be a bounded function. If $E$ is $m$-rectifiable, then
%  \begin{align}
%    \mathscr M_u^k(E)=\int_Eu^{n-m}(x)\operatorname d\cH^k.
%  \end{align}
%  If $E$ is $\cH^k$-rectifiable, then we have $\overline{\mathscr M}_u^k(E)\ge\int_Eu^{n-m}(x)\operatorname d\cH^k$.
%\end{theorem}
%
%The proof of the above theorem is similar to the proof of Theorem 3.2.39 in \cite{Fed69}.

\subsection{Rectifiability and dimension estimates}\label{s:r.d.s}
In this subsection, we let 
$$\{(W^1, p_1, T_1), (W^2, p_2, T_2), \dots,(W^m, p_m, T_m), (W, p, T)\}$$ 
be an $\epsilon_0$-frame and $\displaystyle M=\cap_{i=1}^m\partial W^i$ be an $\epsilon_0$-corner space with $\dim(M)=k=n-m$. In particular, for $m=0$, we let $M=X$. Define 
\begin{align}
	\cS^j(M, W)=\{x\in  M\cap \partial W\colon \text{$T_x(M\cap W)$ is not $(j+1)$-splitting}\}
\end{align}

\begin{lemma}\label{l:S-Theta-dim}  
	$\dim_\cH\left(\cS^{j}(M, W)\right)\le j$. 
\end{lemma}
\begin{proof} 
	This follows by a standard density argument in \cite{BGP} or \cite{CC97-I}. 
\end{proof}

Recall that the singularity angle $\Theta(x, \partial W)=\pi-\beta$, if the tangent cone $T_x(W)=\dR^{n-2}\times C([0,\beta])$,  otherwise $\Theta(x, \partial W)=\pi$. Let $\cS(\Theta,\partial W)=\{x\in \partial W\colon \Theta(x,\partial W)>0\}$. 
The following lemma gives a more refined stratification structure on $\cS(\Theta,\partial W)$, away from a co-dimension 2 subset. 

\begin{lemma}\label{l:frame.sing.direct} Given an $\epsilon_0$-frame $\{(W^1, p_1, T_1), (W^2, p_2, T_2), \dots,(W^m, p_m, T_m), (W, p, T)\}$, there exists a set $\mathcal D\subset X$ with $\dim_\cH(\mathcal D)\le n-2$ so that the following hold for any corner space $\displaystyle M=\cap_{i=1}^m\partial W^i_{s_i}$ and any $0\le t< T$. 
		\begin{enumerate}
		\renewcommand{\labelenumi}{(\roman{enumi})}
		\setlength{\itemsep}{1pt}
		\item $T_x(W_t)$ is $(n-2)$-splitting for any $x\in \cS(\Theta, \partial W_t)\setminus\mathcal D$. 
		\item For any $x\in M\cap\cS(\Theta, \partial W_t)\setminus\mathcal D$, either $\Theta(x,\partial W^i_{s_i})=0$ for every $1\le i\le m$, or $T_x(M\cap W_t)=\dR^{k-1}\times[0,\infty)$. 
	\end{enumerate}
\end{lemma}
%$\mathcal D_t^i=\{x\in  \partial W^i_t\cap W\colon \text{$T_x(\partial W^i_t\cap \partial W)$ is not $(n-2)$-splitting}$\}.
\begin{proof}
  	By Lemma \ref{l:S-Theta-dim}, $\dim_\cH\left(\cS^{n-3}(X, W_t)\right)\le n-3$ and $\dim_\cH\left(\cup_{0\le t< T}\cS^{n-3}(X,W_t)\right)\le n-2$. Away from $\cup_{0\le t< T}\cS^{n-3}(X,W_t)$, we can assume $T_x(W_t)=\dR^{n-2}\times C([0,\beta])$ with $0<\beta<\pi$ and $T_x(W^i)=\dR^{n-2}\times C([0,\alpha_i])$ with $0<\alpha_i\le\pi$ for every $i$. For $\alpha_i<\pi$, let $\mathcal D_s^i=\cS^{n-3}(\partial W^i_s, W)$. By Lemma \ref{l:S-Theta-dim} again, we have $\dim_\cH\left(\mathcal D_s^i\right)\le n-3$ and $\dim_\cH\left(\cup_{s, \alpha_i<\pi}\mathcal D_s^i\right)\le n-2$. 
  	
  Now let $x\in M\cap\cS(\Theta, \partial W_t)$. Away from the co-dimension 2 set $\cup_{s, \alpha_i<\pi}\mathcal D_s^i$ we have either $\alpha_i=\pi$ or $T_x(\partial W^i_{s_i}\cap W_t)$ is $(n-2)$-splitting. We claim that there doesn't exist $i\neq j$ so that $\alpha_i<\pi$ and $\alpha_j<\pi$. If not, then there exist $i\neq j$ and $x\in\partial W^i_{s_i}\cap\partial W^j_{s_j}\cap\partial W_t$, so that both $T_x(\partial W^i_{s_i}\cap W_t)$ and $T_x(\partial W^j_{s_j}\cap W_t)$ are $(n-2)$-splitting. 
  	By the almost orthogonal property of $\epsilon_0$-frame, this implies $\beta=\pi$, a contradiction. 
  	
  	 By the claim, there exists at most one $i$ so that $T_x(\partial W^i_{s_i}\cap W_t)$ is $(n-2)$-splitting, and for every $j\neq i$, we have $\alpha_j=\pi$ and $T_x(W^i_{s_i}\cap W_t)=\dR^{n-1}\times[0,\infty)$. By induction, it's easy to show that $T_x(M\cap W_t)$ is $(k-1)$-splitting. However, $\dim(M\cap W_t)=k$. Then we have that $T_x(M\cap W_t)=\dR^{k-1}\times[0,\infty)$. 
\end{proof}

Define $\Theta^M(x,\partial W)=\pi-\beta$ if $T_x(M\cap W)=\dR^{k-2}\times C([0,\beta])$, otherwise, define $\Theta^M(x,\partial W)=\pi$. Let $\cS^M(\Theta,\partial W)=\{x\in M\cap \partial W\colon \Theta^M(x,\partial W)>0\}$.

\begin{lemma}\label{l:int-ext-theta} There exists $c(n)>0$ small so that the following hold for any $0<\epsilon_0<c(n)$. 
	If $\Theta(x,\partial W^i)=0$ for every $i$, then 
	\begin{align}
		\Theta(x,\partial W)\le \Theta^M(x,\partial W)\le (1+o(\epsilon_0))\Theta(x,\partial W).
		\label{l:int-ext-theta.e1}
	\end{align}
	Let $\mathcal D$ be the set defined in Lemma \ref{l:frame.sing.direct}. Then we have
		\begin{align}
		\cS^M(\Theta,\partial W)\setminus\mathcal D=\cS(\Theta,\partial W)\cap M\setminus\mathcal D.
		\label{l:int-ext-theta.e2}
	\end{align}	
\end{lemma}
\begin{proof}
	Let $T_x(W)=\dR^{n-2}\times C([0,\beta])$ and $T_x(M\cap W)=\dR^{k-2}\times C([0,\beta'])$. Because the intersection angles are all cute, we have $\beta'\le\beta$. Direct computation shows that $\tan\frac{\beta'}{2}=\tan\frac\beta2\cdot(1-o(\epsilon_0))$. Then (\ref{l:int-ext-theta.e1}) follows. 
	
		Let $\cS_\epsilon(\Theta,\partial W)=\{x\in\partial W\colon \Theta(x,\partial W)\ge\epsilon\}$ and $\cS^M_\epsilon(\Theta,\partial W)=\{x\in M\cap \partial W\colon \Theta^M(x,\partial W)\ge\epsilon\}$. By Lemma \ref{l:frame.sing.direct} and (\ref{l:int-ext-theta.e1}), we have
	\begin{align}
		\cS_{2\epsilon}(\Theta,\partial W)\cap M\setminus\mathcal D
		\subseteq\cS^M_\epsilon(\Theta,\partial W)\setminus\mathcal D
		\subseteq\cS_\epsilon(\Theta,\partial W)\cap M\setminus\mathcal D
	\end{align}
for $0<\epsilon<\epsilon_0$ small.
Then the desired result follows. 
\end{proof}

%\begin{corollary}\label{c:frame.sing.direct}
%	There exists a set $\mathcal D\subset X$ with $\dim_\cH(\mathcal D)\le n-2$ so that the following hold for any $x\in \cS^M(\Theta, \partial W)\setminus\mathcal D$. \red{(ii) is weaker than Lemma \ref{l:frame.sing.direct}}
%	\begin{enumerate}
%		\renewcommand{\labelenumi}{(\roman{enumi})}
%		\setlength{\itemsep}{1pt}
%		\item $T_x(W)$ is $(n-2)$-splitting.
%		\item Either $\Theta^M(x,\partial W^i)=0$ for every $i$, or $T_x(M\cap W)=\dR^{k-1}\times[0,\infty)$. 
%	\end{enumerate}
%\end{corollary}

\begin{lemma}\label{l:etrm-cone}
	Let $M$ be an extremal subset in $X\in\Alexnk$ and $p\in M$. For any $\epsilon>0$, there exists an extremal subset $Z\subseteq \Sigma_p(X)$, $p^*\in C(Z)$ and $r=r(p,\epsilon)>0$ so that
	$d_{GH}(B_s(p), B_s(p^*))<\epsilon s$ for any $0<s\le r$, with respect to their extrinsic metrics.
\end{lemma}
\begin{proof}
	It is known that the tangent cone $T_x(M)$ is a metric cone over an extremal subset $Z\subseteq \Sigma_x(X)$. Then the result follows by choosing $r$ as the smallest bad scale at $x$, defined as in \cite{LiNab20}.
\end{proof}

\begin{lemma}\label{l:S-Theta-rect} Let $\mathcal D$ be the set defined in Lemma \ref{l:frame.sing.direct}. Then $\cS^M(\Theta,\partial W)\cap M\setminus\mathcal D$ is $\cH^{k-2}$-rectifiable. 
\end{lemma}

\begin{proof} 
    By Lemma \ref{l:S-Theta-dim} and (\ref{l:int-ext-theta.e2}), it suffices to show that
    $$U_\epsilon=\cS^{M}_\epsilon(\Theta,\partial W)\setminus (\cS^{k-3}(M, W)\cup \mathcal D)$$ 
    is countably $(k-2)$-rectifiable for every $\epsilon>0$ small. 
   
   Let $x\in U_\epsilon$ and $r_x\in(0,1]$ be the radii selected in Lemma \ref{l:etrm-cone}. Define
   $$\Gamma_x^{\,t}=B_{t/2}(x)\cap \{y\in U_\epsilon\colon r_y>2t\}.$$ 
   By definition, tangent cone $T_x(M\cap W)=\dR^{k-2}\times C([0,\beta'])$, where $0<\beta'\le\pi-\epsilon$.  
  Because $x\in M\cap \partial W$ is a boundary point, we have that $T_x(M\cap\partial W)$ is $(k-2)$-splitting.

  Construct a splitting map $u\colon M\cap W\to \dR^{k-2}$ in the same way as in \cite{LiNab20}, where the splitting directions $\dR^{k-2}$ are all in $T_x(M\cap \partial W)$. By Lemma \ref{l:etrm-cone}, within good scale $r_x$ and using the geometry of extremal subsets, the same splitting theory in \cite{LiNab20} applies. In particular, for any $\delta>0$, there exists $r_x>0$ small so that if $y\in B_{r_x}(x)\cap U_\epsilon$ is away from the $\dR^{k-2}$-splitting direction, then $T_y(M\cap W)=\dR^{k-2}\times C([0,\alpha'])$ with $\alpha'\ge\pi-\delta$. Thus $y\notin \cS^{k-2}_\delta(M, W)$. 
  
  Using this property and the same argument as in \cite{LiNab20}, we have that $u\big|_{\Gamma_x^{\,t}}$ is bi-Lipschitz onto from $\Gamma_x^{\,t}$ to its image in $\dR^{k-2}$. Therefore, for any fixed $t>0$, 
  $$\{y\in U_\epsilon\colon r_y>2t\}=\cup_{x\in B_1(p)}\Gamma_x^{\,t}$$
  is countably $(k-2)$-rectifiable. Thus
  $$U_\epsilon=\cup_{t>0}\{y\in U_\epsilon\colon r_y>2t\}$$
  is countably $(k-2)$-rectifiable.
\end{proof}

The above lemma shows that (\ref{s:mono.e9}) holds for $E=\cS^M(\Theta, \partial W)\cap M$. 

\begin{corollary}\label{c:min-cont-Theta}
\begin{align}
	\liminf_{r\to 0^+}\frac{\cH^{k-1}\left(\cup_{x\in \cS^M(\Theta, \partial W)\cap M}B_{r\cdot \Theta^M(x,\partial W)}(x,M\cap\partial W)\right)}{\omega(1)r}
	\ge \int_{\cS^M(\Theta, \partial W)\cap M}\Theta^M(x,\partial W)\operatorname d\cH^{k-2}.
	\label{c:min-cont-Theta.e1}
\end{align}
Since $\Theta^M(x,\partial W)=0$ if $x\notin \cS^M(\Theta,\partial W)$, we rewrite the above formula as
\begin{align}
	\liminf_{r\to 0^+}\frac{\cH^{k-1}\left(\cup_{x\in M\cap\partial W}B_{r\cdot \Theta^M(x,\partial W)}(x,M\cap\partial W)\right)}{\omega(1)r}
	\ge \int_{M\cap\partial W}\Theta^M(x,\partial W)\operatorname d\cH^{k-2}.
	\label{c:min-cont-Theta.e2}
\end{align}
\end{corollary}

\subsection{Monotonic formula} Let the assumptions be the same as in Subsection \ref{s:r.d.s}. Let $\mathbf P\colon \partial W_{t+h}\to \partial W_{t}$ be the closest-point projection map. That is, $y=\mathbf P(x)$ is a point on $\partial W_{t}$ so that $d(y,x)=h$. If $\lambda<0$, then 
\begin{align}
	d(x_1, x_2)\le e^{\lambda h}d(\mathbf P(x_1), \mathbf P(x_2))\le d(\mathbf P(x_1), \mathbf P(x_2)).
	\label{subs:m.formula.e1}
\end{align}
This implies that $\cH^{n-1}(\partial W_{t})-\cH^{n-1}(\partial W_{t+h})\ge 0$. We will show that the difference is related to the total singularity angle $\Theta(x,\partial W_t)$ on $\partial W_t$. For instance, we have
	\begin{align}
	-\frac{\operatorname d\cH^{n-1}(\partial W_t)}{\operatorname d t}
	\ge c\cdot\int_{\partial W_t}\Theta(x,\partial W_t)\,\operatorname d\cH^{n-2}.
\end{align}
 We establish such a formula on corner spaces.

\begin{lemma}[Monotonic formula on corner spaces]\label{l:mono.corner}
    Let $2\le k\le n$. In the distribution sense, we have 
	\begin{align}
		-\frac{\operatorname d\cH^{k-1}(M\cap\partial W_t)}{\operatorname d t}
		\ge c\cdot\int_{M\cap\partial W_t}\Theta^M(x,\partial W_t)\,\operatorname d\cH^{k-2}.
		\label{l:mono.corner.e1}
	\end{align}
\end{lemma}
\begin{proof}
	Not losing generality, we assume $t=0$. By Lemma \ref{l:S-Theta-dim}, we only need to integrate over points $x\in M\cap\partial W_0$ for which $T_x(M\cap W_0)=\dR^{k-2}\times C([0,\beta])$ with $\beta>0$. 
	
    Consider $\mathbf P_t\colon M\cap\partial W_t\to M\cap\partial W_0, q\mapsto q'$, where $q'\in M\cap\partial W_0$ is a point so that $d(q, q')=d(q, M\cap\partial W_0)$.
	%By (\ref{l:1-mono-st-e1}), we have
	%\begin{align}
	%    h\le d(q, M\cap\partial W_0)\le(1+10\epsilon)h+o(h).
	%  \end{align}
We claim that
\begin{align}
	\mathbf P_t(M\cap\partial W_t)\subseteq M\cap\partial W_0\setminus\bigcup_{x\in M\cap\partial W_0}B_{ct\cdot \Theta^M(x,\partial W_0)}(x),
	\label{l:mono.corner.e2}
\end{align}
for $t>0$ sufficiently small. 
Let $x\in M\cap\partial W_0$ with $\Theta^M(x,\partial W_0)>0$ and $y\in M\cap\partial W_t$. Let $z\in M\cap\partial W_0$ so that $d_X(y, z)=d_X(y, M\cap\partial W_0)=t'\ge t$. The following estimate implies (\ref{l:mono.corner.e2}).
\begin{align}
	d_X(x,z)\ge\frac1{10}\tan\left(\frac1{2}\Theta^M(x,\partial W_0)\right)\cdot t'
	\label{l:mono.corner.e3}
\end{align} 
For simplicity, we prove it for $X\in\Alex^n(0)$. For $X\in\Alex^n(-1)$, same proof applies with constant $c$ adjusted accordingly.
 
By assumptions, the intersection angles of $\partial W^i$ and $\partial W^j$, $i\neq j$ and $\partial W$, are no more than $\frac\pi2$. Thus $M\cap \partial W_0$ is an extremal subset in $M\cap W_0$. Though $M\cap W_0$ is not known to be an Alexandrov space, the comparison theory still holds in its small neighborhood and for quasi-geodesics on $M\cap\partial W_0$. 

Suppose 
\begin{align}
	d_X(x,z)<\frac1{10}\tan\left(\frac1{2}\Theta^M(x,\partial W_0)\right)\cdot t'.
	\label{l:mono.corner.e4}
\end{align}
Because intrinsic geodesic connecting $x$ and $z$ in $M\cap\partial W_0$ is a quasi-geodesic, we have 
\begin{align}
	\lim_{t\to 0}\frac{d_{M\cap\partial W_0}(x,z)}{d_X(x,z)}=1
	\label{l:mono.corner.e5}
\end{align}
and
\begin{align}
	d^2_X(x,y)\le d^2_{M\cap\partial W_0}(x,z)+d_X(y,z)^2.
	\label{l:mono.corner.e6}
\end{align}
Therefore, for $t'>0$ sufficiently small, 
\begin{align}
	d^2_X(x,y)\le 2d^2_X(x,z)+(t')^2<(t')^2\sec^2\left(\frac1{2}\Theta^M(x,\partial W_0)\right).
	\label{l:mono.corner.e7}
\end{align}
Let $\eta\in\Sigma_x(M\cap\partial W_0)$ so that
\begin{align}
	d(\uparrow_x^y, \eta)=d(\uparrow_x^y, \Sigma_x(M\cap\partial W_0))\le\frac{\beta}{2}=\frac{\pi}{2}-\frac12\Theta^M(x,\partial W_0).
	\label{l:mono.corner.e8}
\end{align}
Let $\gamma(s)\subseteq M\cap\partial W_0$ be a quasi-geodesic with $\gamma'(0)=\eta$. Let $s_0=t'\cdot \tan\left(\frac1{2}\Theta^M(x,\partial W_0)\right)$. By (\ref{l:mono.corner.e7}) and (\ref{l:mono.corner.e8}), we have
\begin{align}
	d^2_X(y,\gamma(s_0))
	&\le s_0^2+d_X^2(x,y)-2s_0d_X(x,y)\cos(d(\uparrow_x^y, \eta))
	\notag\\
	&< (t')^2\cdot \tan^2\left(\frac1{2}\Theta^M(x,\partial W_0)\right)
	+(t')^2\sec^2\left(\frac1{2}\Theta^M(x,\partial W_0)\right)
	\notag\\
	&\qquad -2t'\cdot \tan\left(\frac1{2}\Theta^M(x,\partial W_0)\right)
	\cdot t'\sec\left(\frac1{2}\Theta^M(x,\partial W_0)\right)
	\cdot \sin\left(\frac1{2}\Theta^M(x,\partial W_0)\right)
	\notag\\
	&=(t')^2.
	\label{l:mono.corner.e9}
\end{align}
This contradicts to the assumption $d_X(y,M\cap\partial W_0)=t'$. Thus (\ref{l:mono.corner.e3}) and (\ref{l:mono.corner.e2}) hold. 

%By Remark \ref{r:covex-cov}, we have $d_{M\cap\partial W_0}$ is $(-\epsilon_0^{-1})$-concave in $M\cap W$ for $\epsilon_0>0$ small. Thus $\mathbf P_t$ is a $1$-coLipschitz map for $t>0$. 

By standard rescalling and density arguments, we have that
\begin{align}
	\cH^{k-1}(M\cap\partial W_t)\le \cH^{k-1}(\mathbf P_t(M\cap\partial W_t))+o(t). 
\end{align}
Combing this with (\ref{l:mono.corner.e2}), we have
\begin{align}
	&\cH^{k-1}(M\cap\partial W_0)-\cH^{k-1}(M\cap\partial W_t)
	\notag \\ 
	& \ge \cH^{k-1}(M\cap\partial W_0\setminus \mathbf P_t(M\cap\partial W_t))-o(t)
	\notag \\
	&\ge \cH^{k-1}\left(\bigcup_{x\in M\cap\partial W_0}B_{ct\cdot\Theta^M(x,\partial W_0)}(x,M\cap\partial W_0)\right)-o(t).
	\label{l:mono.corner.e10}
\end{align}
  Divided by $t$ on the both sides of (\ref{l:mono.corner.e10}) and let $t\to 0^+$. The desired result follows by Corollary \ref{c:min-cont-Theta}.
\end{proof}

By Lemma \ref{l:int-ext-theta} and Lemma \ref{l:mono.corner}, we have the following corollary. 

\begin{corollary}\label{c:mono.corner}
	Let $2\le k\le n$. In the distribution sense, we have 
	\begin{align}
		-\frac{\operatorname d\cH^{k-1}(M\cap\partial W_t)}{\operatorname d t}
		\ge c\cdot\int_{M\cap\partial W_t\setminus\mathcal D}\Theta(x,\partial W_t)\,\operatorname d\cH^{k-2}, 
		\label{c:mono.corner.e1}
	\end{align}
where $\mathcal D$ is the co-dimension 2 subset defined in Lemma \ref{l:frame.sing.direct}.
\end{corollary}

\section{Stratified Integration}\label{s:sum.sing}

%Since we consider co-dimension 1 measure, we can omit all singular points and assume that the tangent cone at every point in $X$ is $\dR^n$. 

%Let $H\subset X$ be a convex level set and $x\in H$ (see Definition \ref{d:cvx.ann}). By the construction in Lemma \ref{l:cox-ann}, the corresponding convex annuli $W$ is induced by $H$ and $H=\partial W$. If the tangent cone $T_x(W)=\dR^{n-2}\times C([0,\beta])$, we define the value of singularity $\Theta(x, H)=2(\pi-\beta)$, otherwise we let $\Theta(x, H)=2\pi$. For a sequence of convex level sets $\{H^1, H^2, \dots, H^n\}$ and $x\in\cap_{1\le i\le m} H^i$, we let 

For a sequence of $\epsilon$-frame $\{W^1, W^2, \dots, W^n\}$ and $x\in\cap_{1\le i\le m} \partial W^i$, we let 
\begin{align}
	\Theta^{\{\partial W^1, \partial W^2, \dots, \partial W^m\}}(x)=\sum_{i=m+1}^n\Theta(x, \partial W^i).
	\label{s:sum.sing.e1}
\end{align}
In particular, we have inductive formula
\begin{align}
	\Theta^{\{\partial W^1, \partial W^2, \dots, \partial W^m\}}(x)=\Theta^{\{\partial W^1, \partial W^2, \dots, \partial W^{m+1}\}}(x)+\Theta(x, \partial W^{m+1}).
	\label{s:sum.sing.e2}
\end{align} 
We will not consider $\Theta^{\{\partial W^1, \partial W^2, \dots, \partial W^{n-1}\}}(x)$ because we need $\dim(\cap_{1=1}^m \partial W^{i})\ge 2$ in the monotonic formula Corollary \ref{c:mono.corner}. Thus we always require $m\ge n-3$ in (\ref{s:sum.sing.e2}) and set
\begin{align}
	\Theta^{\{\partial W^1, \partial W^2, \dots, \partial W^{n-2}\}}(x)=\Theta(x, \partial W^{n-1})+\Theta(x, \partial W^{n})
	\label{s:sum.sing.e3}
\end{align} 
as the initial case. 

Let $\phi_t(x)$ be the re-parameterized gradient flow (or called level set flow) of $d_{\partial W}|W$ so that $\phi_t(x)\in\partial W_t$. It's clear that 
\begin{align}
	\cK(x, \partial W)=\cL_{\nabla \phi_t}\left(\Theta(x,\partial W_t)\right)\,\Big|_{t=0}
\end{align}
in distribution sense, where $\cL$ is the Lie derivative.

%For Alexandrov spaces, we have $\displaystyle\liminf_{t\to 0^+}T_{\phi_t^W(x)}(W_t)\ge T_x(W).$ This implies
%$$\displaystyle\limsup_{t\to 0^+}\Theta (\phi_t^W(x),\,\partial W_t)\le \Theta (x,\partial W).$$ 
%Therefore, $\displaystyle\lim_{t\to 0^+}\Theta (\phi_t^W(x),\,\partial W_t)=0$ if $\Theta (x,\partial W)=0$. 

%For $x\in\partial W$, we define the following in distribution sense
%\begin{align}
%	 \tilde\cK(x, \partial W) &= \renewcommand{\arraystretch}{1.5}
%	\left\{\begin{array}{@{}l@{\quad}l@{}}
%		\displaystyle\lim_{t\to 0^+}\frac{\Theta (\phi_t^W(x),\,\partial W_t)}{t}, 
%			& \text{ if } \Theta(x,\partial W)=0,
%		\\
%		0, & \text{otherwise.}
%	\end{array}\right.
%\label{s:sum.sing.e4}
%\end{align}

\begin{remark}\label{r:flow.def}
	Let us point out that in the definition of $\cK(x)$, if we use the gradient flow of $d_{\partial W}$ instead of the re-parameterized one $\phi_t$, then the corresponding $\tilde\cK=|\nabla d_{\partial W}|^2\cdot\cK$, because $\nabla\phi_t=\nabla d_{\partial W}/|\nabla d_{\partial W}|^2$. Note that in our construction $|\nabla d_{\partial W}|\le 2$. If Theorem A holds for $\cK$, then it also holds for $\tilde\cK$. In our proof, we prefer the level set flow rather than the gradient flow, because it preserves the concavity of the level set. 
\end{remark}

Similarly,  define 
\begin{align}
	\cK^{\{\partial W^1, \partial W^2, \dots, \partial W^m\}}(x)=\sum_{i=m+1}^n\cK(x, \partial W^i).
	\label{s:sum.sing.e5}
\end{align}
Then we have inductive formula
\begin{align}
	\cK^{\{\partial W^1, \partial W^2, \dots, \partial W^m\}}(x)=\cK^{\{\partial W^1, \partial W^2, \dots, \partial W^{m+1}\}}(x)+\cK(x, \partial W^{m+1})
	\label{s:sum.sing.e6}
\end{align} 
for $m\ge n-3$ and the initial case
\begin{align}
	\cK^{\{\partial W^1, \partial W^2, \dots, \partial W^{n-2}\}}(x)=\cK(x, \partial W^{n-1})+\cK(x, \partial W^{n}).
	\label{s:sum.sing.e7}
\end{align} 
Recall that
\begin{align}
	\cK(x)=\inf_{\{W^1, W^2, \dots, W^n\}}\cK^{\{\varnothing\}}(x)=\inf_{\{W^1, W^2, \dots, W^n\}}\sum_{i=1}^n\cK(x, \partial W^i)
	\label{s:sum.sing.e8}
\end{align}
where the infinitum is taken over all family of frames $\{W^1, W^2, \dots, W^n\}$.

\begin{lemma}[Integral formula for $\cK(x)$ on corner spaces]\label{l:int.formula.K} Let 
	$$\{(W^1, p_1, T_1), (W^2, p_2, T_2), \dots,(W^m, p_m, T_m), (W, p, T)\}$$ 
	be an $\epsilon$-frame and $\displaystyle M=\cap_{i=1}^m\partial W^i$ be an $\epsilon$-corner space with $\dim(M)=k=n-m\ge 2$ and $\diam(M)\le 1$. Then
	\begin{align}
		&\int_0^T\int_{M\cap\partial W_t}\cK(x,\partial W_t)\operatorname d\cH^{k-2}_t dt 
		\le  c(n)\cdot \diam(M\cap W)^{k-2},
		\label{l:int.formula.K.e1}
	\end{align}
Since $||\nabla_xd_{\partial W_t}|-1|<\epsilon$, by co-area formula, we can rewrite it in the following form, with an abuse of notation
	\begin{align}
	&\int_{M\cap W}\cK(x,\partial W_x)\operatorname d\cH^{k-1} 
	\le  c(n)\cdot \diam(M\cap W)^{k-2}.
	\label{l:int.formula.K.e2}
\end{align}
\end{lemma}
\begin{proof}
	Let $\cL_X(\omega)$ be the Lie derivative of the differential form $\omega$. Note that 
	\begin{align}
		\int_{M\cap\partial W_{s}}\Theta(x,\partial W_{s})\operatorname d\cH^{k-2}_{s}
		\ge
		\int_{\phi_s\left(M\cap\partial W_0\setminus \cS^M(\Theta,\partial W_0)\right)}\Theta(x,\partial W_{s})\operatorname d\cH^{k-2}_{s}.
	\end{align}
	By the definition of $\cK(x,\partial W)$ and the Leibniz integral formula, we have 
	\begin{align}
		&\frac{\operatorname d}{\operatorname d s}\int_{M\cap\partial W_s}\Theta(x,\partial W_s)\operatorname d\cH^{k-2}_s\,\Big|_{s=0}
		\notag\\
		&\qquad\ge\frac{\operatorname d}{\operatorname d s}\int_{\phi_s\left(M\cap\partial W_0\setminus \cS^M(\Theta,\partial W_0)\right)}\Theta(x,\partial W_s)\operatorname d\cH^{k-2}_s\,\Big|_{s=0}
		\notag \\
		&\qquad =\int_{M\cap\partial W_0\setminus \cS^M(\Theta,\partial W_0)}\cL_{\nabla d_{\phi_s}}\left(\Theta(x,\partial W_s)\operatorname d\cH^{k-2}_s\right)\,\Big|_{s=0}
		\notag \\
		&\qquad =\int_{M\cap\partial W_0\setminus \cS^M(\Theta,\partial W_0)}\cL_{\nabla \phi_s}\left(\Theta(x,\partial W_s)\right)\,\Big|_{s=0}\operatorname d\cH^{k-2}_0
		\notag \\
		&\hskip 1in+\int_{M\cap\partial W_0\setminus \cS^M(\Theta,\partial W_0)}\Theta(x,\partial W_0)\cL_{\nabla \phi_s}\left(\operatorname d\cH^{k-2}_s\right)\,\Big|_{s=0}
		\notag \\
		&\qquad = \int_{M\cap\partial W_0\setminus \cS^M(\Theta,\partial W_0)}\cK(x,\partial W_0)\operatorname d\cH^{k-2}_0+0
		\notag\\
		&\qquad =\int_{M\cap\partial W_0}\cK(x,\partial W_0)\operatorname d\cH^{k-2}_0.
		\label{l:int.formula.K.e3}
	\end{align}
Replacing $\partial W_0$ by $\partial W_t$ and integrating over $t$, we get
	\begin{align}
	&\int_0^{T'}\int_{M\cap\partial W_t}\cK(x,\partial W_t)\operatorname d\cH^{k-2}_t dt 
	\notag\\
	&\qquad \le \int_{M\cap\partial W_{T'}}\Theta(x,\partial W_{T'})\operatorname d\cH^{k-2}_{T'}
	-\int_{M\cap\partial W_0}\Theta(x,\partial W_0)\operatorname d\cH^{k-2}_0
	\notag\\
	&\qquad \le \int_{M\cap\partial W_{T'}}\Theta(x,\partial W_{T'})\operatorname d\cH^{k-2}_{T'}.
	\label{l:int.formula.K.e4}
\end{align}
%	\begin{align}
%	&\int_0^T\int_{M\cap\partial W_t}\cK(x,\partial W_t)\operatorname d\cH^{k-2}_t dt 
%	\notag\\
%	&\qquad \le \int_{\phi_T^{W}\left(M\cap\partial W_0\setminus \cS^M(\Theta,\partial W_0)\right)}\Theta(x,\partial W_T)\operatorname d\cH^{k-2}_T
%	-\int_{M\cap\partial W_0\setminus \cS^M(\Theta,\partial W_0)}\Theta(x,\partial W_0)\operatorname d\cH^{k-2}_0
%	\notag \\
%	&\qquad = \int_{\phi_T^{W}\left(M\cap\partial W_0\setminus \cS^M(\Theta,\partial W_0)\right)}\Theta(x,\partial W_T)\operatorname d\cH^{k-2}_T
%	-0
%	\label{l:int.formula.K.e4.1} \\
%	&\qquad \le \int_{M\cap\partial W_T}\Theta(x,\partial W_T)\operatorname d\cH^{k-2}_T.
%	\label{l:int.formula.K.e4}
%\end{align}
For the above inequality, one can find a simple version in Example \ref{e:2disk}. 
Applying the monotonic formula Corollary \ref{c:mono.corner}, we get 
	\begin{align}
	&\int_0^{T'}\int_{M\cap\partial W_t}\cK(x,\partial W_t)\operatorname d\cH^{k-2}_t dt 
	\le -\frac{\operatorname d\cH^{k-1}(M\cap\partial W_t)}{\operatorname d t}\Big|_{t=T'}.
   \end{align}
Choose $T'<T$ and arbitrarily close to $T$ so that 
	\begin{align}
	-\frac{\operatorname d\cH^{k-1}(M\cap\partial W_t)}{\operatorname d t}\Big|_{t=T'}
	&\le c(n)\cdot \diam(M\cap\partial W_{T'})^{k-2}
	\notag \\
	&\le c(n)\cdot \diam(M\cap W)^{k-2}.
\end{align}
Now we have 
\begin{align}
	&\int_0^{T'}\int_{M\cap\partial W_t}\cK(x,\partial W_t)\operatorname d\cH^{k-2}_t dt 
	\le c(n)\cdot \diam(M\cap W)^{k-2}.
\end{align}
Because $k\ge 2$, the desired result follows by letting $T'\to T$.
%we can omit the integration $\int_{M\cap\partial W_T}\cK(x,\partial W_T)\operatorname d\cH^{k-2}_T$ in the above double integration.

 %$k\ge 2$
%At last, because $\{W^1, W^2, \dots, W^m, W\}$ is an $\epsilon$-frame, we have 
%$$\diam(M\cap \partial W_T)\le 10\cdot\diam(M\cap W).$$ 
\end{proof}

Theorem A is the special case of $m=0$ in the following theorem, by choosing $\epsilon=c(n)>0$ small. 

 \begin{theorem}\label{l:int.corner} Let  $\{W^1, W^2, \dots, W^m\}$
 	be an $\epsilon$-frame and $\displaystyle M=\cap_{i=1}^m\partial W^i$ be an $\epsilon$-corner space with $k=\dim(M)=n-m$ and $\diam(M)\le 1$. Then
	\begin{align}
		\int_{M}\cK^{\{\partial W^1, \partial W^2, \dots, \partial W^m\}}(x) \operatorname d\cH^{k-1}\le c(n,\nu,\epsilon)\cdot \diam(M)^{k-2}.
		\label{l:int.corner.e0}
	\end{align}
\end{theorem}

\begin{proof} We first prove for $k=2$. That is, $m=n-2$ and
	\begin{align}
		\int_{M}\cK(x, \partial W_x^{n-1})+\cK(x, \partial W_x^{n}) \operatorname d\cH^{1}\le c(n).
		\label{l:int.corner.e1}
	\end{align}

By Lemma \ref{t:cov-angle}, we cover $M$ by $\epsilon$-convex annuli $\{(W^{n-1,\ell_1},p^{n-1}_{\ell_1})\}_{\ell_1=1}^{c(n,\nu,\epsilon)}$ so that $p^{n-1}_{\ell_1}\in M$ and 
$$\{W^1, W^2, \dots, W^{n-2}, W^{n-1,\ell_1}\}$$ 
is an $\epsilon$-frame for each $\ell_1$. By Lemma \ref{l:int.formula.K}, we have 
	\begin{align}
	&\int_{M\cap W^{n-1,\ell_1}}\cK(x,\partial W_x^{n-1,\ell_1})\operatorname d\cH^1 
	\le  c(n).
	\label{l:int.corner.e2}
\end{align}
By Lemma \ref{t:cov-angle} again, for each $\ell_1$, cover $M\cap W^{n-1,\ell_1}$ by $\epsilon$-convex annuli $\{(W^{n,\ell_2},p^{n}_{\ell_2})\}_{\ell_2=1}^{c(n,\nu,\epsilon)}$ so that $p^{n}_{\ell_2}\in M\cap W^{n-1,\ell_1}$ and 
$$\{W^1, W^2, \dots, W^{n-2}, W^{n-1,\ell_1}, W^{n,\ell_2}\}$$ 
is an $\epsilon$-frame for every $\ell_2$. Apply Lemma \ref{l:int.formula.K} to $\{W^1, W^2, \dots, W^{n-2}, W^{n,\ell_2}\}$.
We get 
	\begin{align}
	&\int_{M\cap W^{n,\ell_2}}\cK(x,\partial W_x^{n,\ell_2})\operatorname d\cH^1 
	\le  c(n).
	\label{l:int.corner.e3}
\end{align}
Fix $\ell_1$ and sum up all corresponding $1\le\ell_2\le c(n,\nu,\epsilon)$. We get
	\begin{align}
	&\int_{M\cap W^{n-1,\ell_1}}\cK(x,\partial W_x^{n})\operatorname d\cH^1 
	\le  c(n,\nu,\epsilon).
	\label{l:int.corner.e4}
\end{align}
Inequality (\ref{l:int.corner.e1}) follows by summing up (\ref{l:int.corner.e2}) and (\ref{l:int.corner.e4}), and sum up for all $1\le\ell_1\le c(n,\nu,\epsilon)$. 

Now we prove (\ref{l:int.corner.e0}) by induction, which is similar as the case $k=2$. By Lemma \ref{t:cov-angle}, there exists an $\epsilon$-convex annuli covering $\{(W^{m+1,\ell},p_{\ell})\}_{\ell=1}^{c(n,\nu,\epsilon)}$ for $M$. Applying inductive hypothesis on 
$$\{W^1, W^2, \dots, W^{m}, W_t^{m+1,\ell}\},$$ 
we have
	\begin{align}
	&\int_{M\cap \partial W_t^{m+1,\ell}}\cK^{\{\partial W^1, \partial W^2, \dots, \partial W^{m},\partial W_t^{m+1,\ell}\}}(x)\operatorname d\cH^{k-2} 
	\le  c(n,\nu,\epsilon)\cdot \diam(M\cap\partial W_t^{m+1,\ell})^{k-3}.
	\label{l:int.corner.e5}
\end{align}
Integrating over $t$, we get 
	\begin{align}
	&\int_{M\cap W_x^{m+1,\ell}}\cK^{\{\partial W^1, \partial W^2, \dots, \partial W^{m},\partial W_x^{m+1,\ell}\}}(x)\operatorname d\cH^{k-1} 
	\le  c(n,\nu,\epsilon)\cdot \diam(M\cap\partial W^{m+1,\ell})^{k-2}.
	\label{l:int.corner.e6}
\end{align}
Applying Lemma \ref{l:int.formula.K} to $\{W^1, W^2, \dots, W^{m}, W^{m+1,\ell}\}$, we get 
	\begin{align}
	&\int_{M\cap W_x^{m+1,\ell}}\cK(x,\partial W_x^{m+1,\ell})\operatorname d\cH^{k-1} 
	\le  c(n)\cdot \diam(M\cap W^{m+1,\ell})^{k-2}.
	\label{l:int.corner.e7}
\end{align}
Summing up (\ref{l:int.corner.e6}), (\ref{l:int.corner.e7}) and by the inductive formula (\ref{s:sum.sing.e6}), we have
	\begin{align}
	&\int_{M\cap W_x^{m+1,\ell}}\cK^{\{\partial W^1, \partial W^2, \dots, \partial W^{m}\}}(x)\operatorname d\cH^{k-1} 
	\le  c(n,\nu,\epsilon)\cdot \diam(M\cap W^{m+1,\ell})^{k-2}.
	\label{l:int.corner.e8}
\end{align}
The desired result follows by summing up (\ref{l:int.corner.e8}) for all $1\le\ell\le c(n,\nu,\epsilon)$. 
\end{proof}

\begin{remark}
It follows by the same proof that
	\begin{align}
	\int_{M}\Theta^{\{\partial W^1, \partial W^2, \dots, \partial W^{m}\}}(x) \operatorname d\cH^{k-1}\le c(n,\nu,\epsilon)\cdot \diam(M)^{k-2}.
\end{align}
In particular, 
	\begin{align}
	\int_{B_r}\Theta(x) \operatorname d\cH^{n-1}\le c(n,\nu,\epsilon)\cdot r^{n-2}.
\end{align}
However, this is less interesting than the integral bound of $\cK(x)$, because $\Theta(x)$ doesn't reflect the singularity of $X$. Nevertheless, it shows that the distributional derivative (\ref{s:sum.sing.e4}) is well-defined. 
\end{remark}

At last, let us prove Theorem B. 
\begin{proof} For $n=2$, the continuity follows directly from Gauss-Bonnet formula. Assume $n\ge 3$. Note that for Alexandrov spaces, tangent cones are lower semi-continuous \cite{BGP}. That is, if $(X_i, x_i)\overset{d_{GH}}\longrightarrow (X,x)$,  $T_{x_i}(X_i)=C(\Sigma_{x_i})$ and $T_{x}(X)=C(\Sigma_{x})$, then for any Gromov-Hausdorff limit $\displaystyle\Sigma=\lim_{i\to\infty}\Sigma_{x_i}$, there is a distance non-increasing onto map from $\Sigma$ to $\Sigma_x$. This implies $\displaystyle\lim_{i\to\infty}\int_{B_1(p_i)} \theta_i\, \operatorname d\cH^{n-2}=0$. 
	
For the convergence of $\int_{B_r(x_i)} \cK_i\, \operatorname d\cH^{n-1}$, it suffices to show that for any $x_i\in X_i$ with $x_i\to x\in M$, there exists $r>0$ so that 
\begin{align}
	\lim_{i\to\infty}\int_{B_r(x_i)} \cK_i\, \operatorname d\cH^{n-1}=0.
	\label{t:K.cont.e1.0}
\end{align}
	
It's clear that
		\begin{align}
			\liminf_{i\to\infty}\int_{B_r(x_i)} \cK_i\, \operatorname d\cH^{n-1}
			\ge 0.
			\label{t:K.cont.e1}
		\end{align}
	For any $\delta>0$, there exists $N>0$ so that the following hold for every $i\ge N$. Every $\epsilon$-frame in $M$ can be lifted to a $\frac12(\epsilon-\delta)$-frame in $X_i$, with a small ball of radius $0<\delta\ll\epsilon$ removed around every center of the convex annuli. The choice of $\delta$ is to preserve the good scales in the lifted frames. Let $q^i_j$, $j=1,2,\dots,c(n,\nu,\epsilon)$ be the centers of the lifted frames. Then the lifting can be chosen so that the lifted frames exhaust $B_r(x_i)\setminus \cup_{j=1}^{c(n,\nu,\epsilon)}B_{\delta}(q^i_j)$, where $\delta\to 0$ as $i\to\infty$. Since $M$ is smooth, we can perturb the distance functions and choose $r>0$ small, so that all corresponding frames are smooth in $B_r(x_i)$. In particular, we have $\Theta(y,\partial W_t)=0$ for every convex annulus $W$ in the covering, and $y$ away from the centers of the convex annuli.

	It's easy to see that the lower semi-continuity property also holds for tangent cones on corner spaces. Therefore, $\Theta(x_i,\partial W_i)$ is upper semi-continuous, as $i\to\infty$. By inequality (\ref{l:int.formula.K.e4}) and the convergence of gradient flow, we have that 
	\begin{align}
		\limsup_{i\to\infty}\int_{B_r(x_i)\setminus \cup_{j=1}^{c(n,\nu,\epsilon)}B_{\delta}(q^i_j)} \cK_i\, \operatorname d\cH^{n-1}
		\le 0.
		\label{t:K.cont.e1}
	\end{align}
	By Theorem A, we have that 
	\begin{align}
		\int_{\cup_{j=1}^{c(n,\nu,\epsilon)}B_{\delta}(q^i_j)} \cK_i\, \operatorname d\cH^{n-1} \le c(n,\nu,\epsilon)\delta^{n-2}.
		\label{t:K.cont.e2}
	\end{align}
	The desired result follows by letting $\delta\to 0$. 
	
\end{proof}

\section{Appendix}

As an application of the covering technique in this paper, we give a direct proof for Theorem \ref{t:int.scal}, which is restated below. 

\begin{theorem}[Petrunin \cite{Pet09}]\label{t:Pet.X}
	For any $n$-dimensional manifolds $(X, g)$ with sectional curvature $\sec_X\ge -1$, we have a prior $L^1$-bound for the scalar curvature
	$\dsp\int_{B_1}scal \,\operatorname d {vol}_{g}\le c(n)$.
\end{theorem}

Our proof is built on the frame work established in the previous sections. 
Because $X$ is a smooth manifold and the availability of Gauss–Codazzi formula, the covering annuli need not to be convex. In fact, the good scale annuli covering Theorem \ref{t:ann_cover} is sufficient for our proof. Because the number of such covering annuli doesn't depend on the volume, the resulted estimate for the curvature integral doesn't depend on volume as well. We first give a notion for the so-called flexible $\epsilon$-frame. It is similar to the $\epsilon$-frame in Section \ref{s:annuli}, but without convexity conditions.

The following definition is similar to Definition \ref{d:cvx.ann}.  

\begin{definition}[flexible $(\epsilon,\sigma)$-good scale annulus]\label{d:cvx.ann.mfd} 
	A subset in $X$, denoted by $W_a^b(p)$, is said to be a flexible $(\epsilon,\sigma)$-good scale annulus if there exists a semi-concave function $f\colon X\to \dR$ so that the following hold. 
	\begin{enumerate}
		\renewcommand{\labelenumi}{(\roman{enumi})}
		\item (Defined by $f$) $W_a^b(p)=\cup_{a\le t\le b} H_t$, where $H_t=f^{-1}(t)$. 
		\item (Good scale) $A_{a}^{b}(p)$ is an $(\epsilon,\sigma)$-good scale annulus.
		\item (Distance-like) $(1-\epsilon)d_p(x)\le |f(x)-d_p(x)|\le (1+\epsilon)d_p(x)$ for any $x\in A_{\sigma a}^{\sigma^{-1}b}(p)$
	\end{enumerate}
Note that any $(\frac12\epsilon,\frac12\sigma)$-good scale annulus $A_a^b(p)$ is naturally a flexible $(\epsilon,\sigma)$-good scale annulus, for $f(x)=d_p(x)$. We may denote the flexible good scale annuli $W_a^b(p)$ by the couple $\{W_a^b(p), f\}$ if the corresponding defining function is involved. 
\end{definition}

We will apply Theorem \ref{t:int.scal} recursively on each level set of $f$ and their intersections. Similar to Lemma \ref{t:adj-angle}, we can also adjust these annuli so that their intersection angles are close to $\frac\pi2$. For this purpose, we have the following definition, which is similar to Definition \ref{d:frame}. 

\begin{definition}\label{d:frame.mfd}
	A sequence of flexible $(\epsilon,\epsilon)$-good scale annuli $\{W_{a_i}^{b_i}(p_i), f_i\}$, $i=1,2,\dots,m$ is called a flexible $(\epsilon,\delta)$-frame if 
	\begin{align}
		-2\delta \le \langle \nabla_x f_i, \nabla_xf_j \rangle\le -\delta
	\end{align}
	for any $x\in f_i^{-1}(t) \cap f_j^{-1}(s)$, $i\neq j$, $a_i\le t< b_i$ and $a_j\le s< b_j$. The intersection $\displaystyle M=\cap_{i=1}^m f_i^{-1}(t_i)$ is called a flexible $(\epsilon,\delta)$-corner space of $\{W_{a_i}^{b_i}(p_i), f_i\}$. 
	
	A flexible $(\epsilon,\delta)$-frame is called a flexible $\epsilon$-frame if $\delta=\delta(n,\epsilon)$. We may omit the parameters and denote a flexible frame by $\{W_i\}$ if the omitted parameters are not involved in the context.  
\end{definition}

By the same proof as Lemma \ref{t:adj-angle} and then cover the corner space and adjust the intersection angles recursively, we have the following statement, which is similar to Lemma \ref{t:cov-angle}. 

\begin{lemma}[Stratified covering by flexible good scale annuli]\label{t:cov-angle.mfd} There exists $c(n)>0$ so that the following holds for	$0<10\epsilon<\delta<\frac1{200n}$.  
	Let $\{W_i, h_i\}$, $i=1,2,\dots,m$ be a flexible $(\epsilon,\delta)$-frame and $\displaystyle M=\cap_{i=1}^m h_i^{-1}(t_i)$ be a flexible $(\epsilon,\delta)$-corner space.	
   Then there exists a flexible $(\epsilon,\delta)$-good scale annulus covering $\{W_{a_\ell}^{b_\ell}(p_\ell), f_\ell\}_{\ell=1}^N$ of $M\cap B_1$, so that 
	$$\{W_1, h_1\}, \{W_2, h_2\}, \dots, \{W_m, h_m\},  \{W_{a_\ell}^{b_\ell}(p_\ell), f_\ell\}$$ 
	is a flexible $(\epsilon, 12\delta)$-frame for any $1\le\ell\le N=c(n,\epsilon, \delta)$.
\end{lemma}

Theorem \ref{t:Pet.X} follows from the case $m=0$ in the following statement. 

 \begin{theorem}\label{l:int.corner.mfd} Let  $\{W_i,h_i\}$, $i=1,2,\dots,m$
	be a flexible $\epsilon$-frame and $\displaystyle M=\cap_{i=1}^m h_i^{-1}(t_i)$ be a flexible $\epsilon$-corner space with $\diam(M)\le 1$. Let $K_M^-(x)=\max\{-\min\{\sec_M(x)\},1\}$ and $k=\dim(M)=n-m$. Then
	\begin{align}
		\int_{M\cap B_1(q)} scal_M \operatorname d\text{vol}_M\le c(n,\epsilon)\cdot\left(\diam(M)^{k-2}+\int_{M\cap B_2(q)} K_M^-\right).
		\label{l:int.corner.mfd.e0}
	\end{align}
\end{theorem}

\begin{proof}
	We prove (\ref{l:int.corner.mfd.e0}) by induction on $\dim(M)$. The case of dimension 2 follows from Gauss-Bonnet formula. Let $\{W_{a_\ell}^{b_\ell}(p_\ell), f_\ell\}_{\ell=1}^{c(n,\epsilon)}$ be a flexible $\epsilon$-good scale annuli covering described as in Lemma \ref{t:cov-angle.mfd}. Fix $\ell$ and denote $p=p_\ell$, $a=a_\ell$ and $b=b_\ell$. Let $L_t=M\cap f_\ell^{-1}(t)$ be the next stratum of corner space. Because $X$ is smooth, we can perturb the defining functions $h_i$ and $f_\ell$ so that $M$ and $L_t$ are all smooth. For simplicity of the notation, and for the integration domain only, we will denote $M\cap B_1(q)$ by $M$, and denote  $M\cap B_2(q)$ by $M^{(2)}$. Similarly, we denote $L_t\cap B_1(q)$ by $L_t$, and denote  $L_t\cap B_2(q)$ by $L_t^{(2)}$
	
	Applying the inductive hypothesis on 
	$$\{W^1, W^2, \dots, W^{m}, W_{a}^{b}(p)\},$$ 
	we have
	\begin{align}
		&\int_{L_t}scal_{L_t} 
		\le  c(n,\epsilon)\cdot\left(\diam(L_t)^{k-3}+\int_{L_t^{(2)}} K_{L_t}^-\right),
		\label{l:int.corner.mfd.e5}
	\end{align}
where $K_{L_t}^-(x)=\max\{-\min\{\sec_{L_t}(x)\},1\}$. 
Most of the following estimates are similar to those in Petrunin's proof. To relate $scal_{L_t}(x)$ and $scal_{M}(x)$, we need the following Gauss–Codazzi formula:
\begin{align}
	scal_{L_t}=scal_{M}-2\Ric_M(u,u)+G,
	\label{l:int.corner.mfd.e7}
\end{align}
where $G(x)=\sum_{i\neq j}k_i(x)k_j(x)$ and $k_i(x)$ are the principle curvature of $L_t$ at $x$ in $M$. Let $b'\in[b,2b]$  be determined latter. Integrating (\ref{l:int.corner.mfd.e7}) over $M\cap f_\ell^{-1}[a,b']$ and by co-area formula, we have 
\begin{align}
	\int_{M\cap f_\ell^{-1}[a,b']}scal_{M}=\int_{a}^{b'}\int_{L_t}\frac{scal_{L_t}}{|\nabla_xf_\ell|}+\int_{M\cap f_\ell^{-1}[a,b']}(2\Ric_M(u,u)-G),
	\label{l:int.corner.mfd.e7.5}
\end{align}
We use the same Bochner formula as in Petrunin's proof \cite{Pet09}
\begin{align}
	\int_{M\cap f_\ell^{-1}[a,b']}\Ric_M(u,u)=\int_{M\cap f_\ell^{-1}[a,b']}G+\int_{L_{a}}H-\int_{L_{b'}}H,
	\label{l:int.corner.mfd.e8}
\end{align}
where $H(x)=\sum k_i(x)$ is the mean curvature of $L_t$ in $M$. Plugging (\ref{l:int.corner.mfd.e8}) into (\ref{l:int.corner.mfd.e7.5}), we get 
\begin{align}
	\int_{M\cap f_\ell^{-1}[a,b']}scal_{M}=\int_{a}^{b'}\int_{L_t}\frac{scal_{L_t}}{|\nabla_x f_\ell|}+\int_{M\cap f_\ell^{-1}[a,b']}G+2\int_{L_{a}}H-2\int_{L_{b'}}H,
	\label{l:int.corner.mfd.e9}
\end{align}
It turns out that the right hand side can be controlled by $K_M^-$ and the good scale on $W_{a_\ell}^{b_\ell}(p_\ell)$. 

First of all, by (\ref{l:int.corner.mfd.e7}) again,
\begin{align}
	G=scal_{L_t}-scal_{M}+2\Ric_M(u,u)\le scal_{L_t}+m^2 K_M^-
	\label{l:int.corner.mfd.e9.5}
\end{align}
By the good scale on $f_\ell^{-1}[a,b]$, we have $||\nabla_x f_\ell|-1|<10\epsilon$ and $\dsp k_i(x)\le \frac{2}{d_p(x)}$ for every $x\in L_t$. Th latter one implies that $\dsp H(x)\le \frac{2m}{d_p(x)}$.  Plugging these estimates and (\ref{l:int.corner.mfd.e9.5}) into (\ref{l:int.corner.mfd.e9}), we get 
\begin{align}
	&\int_{M\cap f_\ell^{-1}[a,b']}scal_{M}
	\notag \\
	&\le 3\int_{a}^{b'}\int_{L_t}scal_{L_t}+m^2\int_{M\cap f_\ell^{-1}[a,b']}K_M^-+c(n)\cdot a^{k-2}-2\int_{L_{b'}}H,
	\label{l:int.corner.mfd.e9.6}
\end{align}
We need to find a lower bound of $\dsp\int_{L_{b'}}H$. Because the intersection angles of the level sets are all strictly less than $\frac\pi2$, we can choose the directions of $k_i$ so that 
\begin{align}
	\frac{d(\text{vol}(L_t))}{dt}=\int_{L_t}\frac{H}{|\nabla_x f_\ell|}.
	\label{l:int.corner.mfd.e9.6.5}
\end{align}
By the good scale structure, we have that 
$$\frac{d(\text{vol}(L_t))}{dt}\ge -c(n)\cdot t^{k-2}
$$
for some $t\in[b,2b]$.
Therefore, we can choose $b'\in[b,2b]$ so that 
\begin{align}
	\int_{L_{b'}}H \ge -c(n)\cdot b^{k-2}.
	\label{l:int.corner.mfd.e9.7}
\end{align} 
Plugging this into (\ref{l:int.corner.mfd.e9.6}), we get 
\begin{align}
	\int_{M\cap f_\ell^{-1}[a,b']}scal_{M}&\le 3\int_{a}^{b'}\int_{L_t}scal_{L_t}+c(n) \cdot \int_{M\cap f_\ell^{-1}[a,b']}K_M^-+c(n)\cdot b^{k-2}
	\label{l:int.corner.mfd.e11},
\end{align}
for some $b'\in[b,2b]$. Integrating the inductive hypothesis (\ref{l:int.corner.mfd.e5}) over $t\in[a,b']\subseteq[a,2b]$, we get 
\begin{align}
	&\int_{a}^{b'}\int_{L_t}scal_{L_t}dt 
	\le  c(n,\epsilon)\cdot \left(b^{k-2}+\int_a^{b'}\int_{L_t^{(2)}} K_{L_t}^-\right).
	\label{l:int.corner.mfd.e6}
\end{align}
Therefore,
\begin{align}
	\int_{M \cap f_\ell^{-1}[a,b']}scal_{M}&\le c(n,\epsilon)\cdot \left(b^{k-2}+\int_a^{b'}\int_{L_t^{(2)}} K_{L_t}^-+\int_{M\cap f_\ell^{-1}[a,b']}K_M^-\right),
	\label{l:int.corner.mfd.e11.1}
\end{align}
for some $b'\in[b,2b]$.

It remains to find an upper bound of $\dsp\int_a^{b'}\int_{L_t^{(2)}} K_{L_t}^-$ in terms of $\dsp\int_{M^{(2)}\cap f_\ell^{-1}[a,b']} K_M^-$. By Gauss–Codazzi formula again, 
\begin{align}
	K_{L_t}^-\le K_{M}^-+2(H^-+m\cdot\max\{k_i,0\})\cdot(\max\{k_i,0\}),
	\label{l:int.corner.mfd.e6.1}
\end{align}
where $H^\pm(x)=\max\{\pm H(x),0\}$. Because $\dsp k_i(x)\le\frac{2}{d_p(x)}$ and $||\nabla_x f_\ell|-1|<10\epsilon$, we have 
\begin{align}
	\int_{L_t} H^-
	&\le 2\int_{L_t}\frac{H^-}{|\nabla_x f_\ell|}
	\notag \\
	&= 2\int_{L_t}\frac{H^+}{|\nabla_x f_\ell|}
	-2\int_{L_t}\frac{H}{|\nabla_x f_\ell|}
	\notag \\
	&\le c(n)\cdot t^{k-2}-2\cdot \frac{d(\text{vol}(L_t))}{dt}.
\end{align}
Integrating over $t$, we have
\begin{align}
	\int_{M^{(2)}\cap f_\ell^{-1}[a,b']} (H^-\cdot\max\{k_i,0\})
	\le c(n)\cdot b^{k-2}.
\end{align}
Combining this with (\ref{l:int.corner.mfd.e6.1}), we get
\begin{align}
	\int_a^{b'}\int_{L_t^{(2)}}K_{L_t}^-\le \int_{M^{(2)}\cap f_\ell^{-1}[a,b']}K_{M}^-+c(n)\cdot b^{k-2}
\end{align}
Plugging this into (\ref{l:int.corner.mfd.e11.1}), we have
\begin{align}
	&\int_{M\cap f_\ell^{-1}[a,b']}scal_{M} 
	\le  c(n,\epsilon)\cdot \left(b^{k-2}+\int_{M^{(2)}\cap f_\ell^{-1}[a,b']} K_{M}^-\right).
\end{align}

Note that $scal_M(x)\ge -m^2\cdot K_M^-(x)$ and $b'\in[b,2b]$. We have 
\begin{align}
	\int_{M\cap f_\ell^{-1}[a,b]}scal_{M} 
	&\le  2\cdot c(n,\epsilon)\cdot \left(b^{k-2}+\int_{M^{(2)}\cap f_\ell^{-1}[a,b']} K_{M}^-\right)
	\notag \\
	&\le 2\cdot c(n,\epsilon)\cdot \left(b^{k-2}+\int_{M^{(2)}\cap f_\ell^{-1}[a,2b]} K_{M}^-\right).
\end{align}
The result follows by summing up the above inequality for all $\ell=1,2,\dots,c(n,\epsilon)$. 

\end{proof}

\end{document}